\documentclass[aap,preprint]{imsart}
\usepackage{amsmath}   
\usepackage{amssymb}   
\usepackage{amsthm}    
\usepackage{stmaryrd}  
\usepackage{titletoc}  
\usepackage{mathrsfs}  
\usepackage{color, enumerate}

\theoremstyle{plain}
\newtheorem{thm}{Theorem}[section]
\newtheorem{cor}{Corollary}[section]
\newtheorem{lem}{Lemma}[section]
\newtheorem{prop}{Proposition}[section]
\newtheorem{exam}{Example}[section]

\theoremstyle{definition}
\newtheorem{defn}{Definition}[section]

\newtheorem{rmk}{Remark}[section]

\newcommand{\eps}{\varepsilon}

\DeclareMathOperator*{\esssup}{ess\,sup}
\DeclareMathOperator*{\essinf}{ess\,inf}
\newcommand{\cA}{\mathcal{A}}
\newcommand{\cB}{\mathcal{B}}
\newcommand{\cG}{\mathcal{G}}

\newcommand{\cH}{\mathcal{H}}
\newcommand{\cL}{\mathcal{L}}

\newcommand{\cS}{\mathcal{S}}

\newcommand{\bP}{\mathbb{P}}
\newcommand{\bR}{\mathbb{R}}
\newcommand{\bN}{\mathbb{N}}

\newcommand{\sF}{\mathscr{F}}
\newcommand{\sP}{\mathscr{P}}

\newcommand{\tbf}{\textbf}
\newcommand{\llangle}{\left\langle}
\newcommand{\rrangle}{\right\rangle}

\makeatletter\@addtoreset{equation}{section} \makeatother
 \allowdisplaybreaks

\begin{document}

\begin{frontmatter}

\title{Controlled Reflected SDEs and Neumann Problem for Backward SPDEs}
\runtitle{Neumann Problem for Backward SPDEs}

\begin{aug}
  \author{\fnms{Erhan}  \snm{ Bayraktar}\corref{}\thanksref{t1}\ead[label=e1]{erhan@umich.edu}},
  \author{\fnms{Jinniao} \snm{ Qiu} \thanksref{t2}  \ead[label=e2]{jinniao.qiu@ucalgary.ca}}
 
  \thankstext{t1}{Supported in part by the National Science Foundation under grant DMS-1613170 and the Susan
M. Smith Professorship. }

 \thankstext{t2}{Partially supported by the National Science and Engineering Research Council of Canada and by the start-up funds from the University of Calgary.}

  \runauthor{Bayraktar and Qiu}

  \affiliation{University of Michigan and University of Calgary}

  \address{530 Church Street, Ann Arbor, \\ MI  48109-1109, USA\\ 
          \printead{e1}}

\address{2500 University Drive NW, Calgary, \\ AB T2N 1N4, Canada.\\ 
          \printead{e2}}

 \end{aug}


\begin{abstract}
  We solve the optimal control problem of a one-dimensional reflected stochastic differential equation, whose coefficients can be path dependent. The value function of this problem is characterized by a backward stochastic partial differential equation (BSPDE) with Neumann boundary conditions. We prove the existence and uniqueness of a sufficiently regular solution for this BSPDE, which is then used to construct the optimal feedback control. In fact we prove a more general result: The existence and uniqueness of strong solution for the Neumann problem for general nonlinear BSPDEs, which might be of interest even out of the current context.

\end{abstract}

\kwd[Primary ]{60K35}
\kwd{93E20, 60H15, 91G80}
\end{keyword}


\end{frontmatter}

\renewcommand{\baselinestretch}{1.1}\normalsize

\section{Introduction}

Let $T\in(0,\infty)$ and  $(\Omega,\bar{\sF},\bP)$ be a probability space equipped with a filtration $\{\bar{\sF}_t\}_{0 \leq t \leq T}$ which satisfies the usual conditions.  The filtration $\bar{\sF}$ is generated by two independent $m$-dimensional Brownian motions $W$ and $B$.  We denote by  $\{\sF_t\}_{t\geq 0}$ the filtration generated by $W$, together with all $\bP$ null sets. The predictable $\sigma$-algebra on $\Omega\times[0,+\infty)$ corresponding to $\{\sF_t\}_{t\geq0}$ and $\{\bar{\sF}_t\}_{t\geq0}$ is denoted $\sP$, respectively, $\bar{\sP}$. 

In this paper, we consider the following stochastic optimal control problem:
\begin{align}
  \min_{\theta}
  E\left[\int_0^T\!\! f_t(X_t,\theta_t)\,dt
  +\int_0^T\!\! g_t(X_t)\,dL_t 
  +\int_0^T\!\! g_t(X_t)\,dU_t
  +G(X_T)
  \right]
  \label{min-contrl-probm}
\end{align}
subject to
\begin{equation}\label{state-proces-contrl}
\left\{
\begin{split}
&dX_t=  \beta_t(X_t,\theta_t)\,dt+\sigma_t(X_t)\,dW_t+\bar\sigma_t(X_t)\,dB_t\\
&\quad\quad\quad +dL_t-dU_t,    \quad  t\in[0,T] ;\\
&X_0=x; \quad L_0=U_0=0;\\
&0\leq X_t\leq b, \quad \text{a.s.};\\
&\int_0^T\!\! X_t\,dL_t =\int_0^T(b-X_s)\,dU_s=0,\quad \text{a.s.,}
\end{split}
\right.
\end{equation}
where $L$ and $U$ are two non-decreasing processes. The real-valued process $(X_t)_{t\in[0,T]}$ is the {\sl state process}. Its drift is governed by the {\sl control} $\theta$.  
  We sometimes write $X^{s,x;\theta}_t$ for $0\leq s\leq t\leq T$ to indicate the dependence of the state process on the control $\theta$, the initial time $s$ and initial state $x\in \mathbb{R}$. The set of {\sl admissible controls} consists of all $\bar\sF_t$-adapted processes $\theta $ such that the reflected stochastic differential equation (SDE) \eqref{state-proces-contrl} admits a unique solution and $\theta_t\in\Theta$ a.s for each $t\in[0,T]$ with set $\Theta\subset \bR^n$.

Classical stochastic control problems, see e.g. \cite{Fleming-Rishel-1975,Flem-Soner-2006controlled,kushner1971introduction}, have been generalized more recently to handle the path dependent case \cite{Ekren-Touzi-Zhang-2012,peng1990general,pengwu1999fully}. We will in addition consider the problem of controlling reflected path dependent SDEs. The analysis of such control problems is motivated by the drift rate controlled queueing problem in \cite{ata2005drift}, where the control problem is of ergodic/stationary type and is concerned with minimizing the long-run average cost under the Markovian framework. In contrast to that set-up, the coefficients in \eqref{min-contrl-probm} and \eqref{state-proces-contrl} are allowed to be random and thus can be non-Markovian; more precisely, we assume: 
      \[
   \begin{split}
&(\mathcal{A} 0) \quad \text{The coefficients } \beta,\,\sigma,\,\bar\sigma,\, f,\, g \text{ are } \sP\otimes\cB(\bR)\otimes\cB(\bR^n)\text{-measurable and } \\
& G\text{ is } 
\sF_T\otimes\cB(\bR)\text{-measurable.}
\end{split}
\]
We would also note that, as stated in \cite{ata2005drift}, because the reflecting barriers are not discretionary and only the drift rate is controlled, the control problem does not fall in the spectrum of ``singular" stochastic control. To the best of our knowledge, ours is the first analysis of the controlled reflected SDEs with random coefficients.

   Let us denote the dynamic version of the cost by
\begin{equation}\label{cost0}
\begin{split}
  J_t(X_t;\theta)=
  E \bigg[
  \int_t^T\!\! f_s(X_s,\theta_s)\,ds
  +\int_t^T\!\! g_s(X_s)\,dL_s\\
  +\int_t^T\!\! g_s(X_s)\,dU_s
  +G(X_T)\Big| \bar{\sF}_t
    \bigg],
    \end{split}
\end{equation}
and denote
\begin{equation}\label{value-func}
  u_t(x){\triangleq}\essinf_{\theta} J_t(X_t;\theta)\big|_{X_t=x}.
\end{equation}
In view of Peng's seminal work \cite{Peng_92} on non-Markovian stochastic optimal control, the dynamic programming principle suggests that the value function $u$ is the first component of the pair $(u,\psi)$ satisfying formally the following Neumann problem for backward stochastic partial differential equation (BSPDE):
\begin{equation}\label{BSPDE-NP}
  \left\{\begin{array}{l}
  \begin{aligned}
  -du_t(x)=\, &\bigg[ \frac{1}{2}\left( | \sigma_t(x)|^2+|\bar\sigma_t(x)|^2\right)D^2 u_t(x)
        + \sigma D \psi_t(x)
        \\
        &\quad+{\mathbb H} _t(x,Du_t(x))
         \bigg]\, dt  
       \  -\psi_t(x)\, dW_{t}, \\
        &\quad\quad\quad\quad
        \quad (t,x)\in [0,T]\times [0,b];\\
        Du_t(0)=\,&g_t(0),\quad Du_t(b)=g_t(b);\\
    u_T(x)=\,& G(x),  \quad x\in[0,b],
    \end{aligned}
  \end{array}\right.
\end{equation}
with Hamiltonian function
\begin{align}
{\mathbb H} _t(x,Du_t(x))\triangleq \essinf_{\theta\in\Theta}\left\{
        \beta_t(x,\theta)D u_t(x)+f_t(x,\theta)\right\},
        \label{eq-hamiltonian}
\end{align}
for $(t,x)\in [0,T]\times [0,b]$.

First, the self-contained proofs for the existence and uniqueness of strong solution are given for the Neumann problem of general nonlinear BSPDEs. \footnote{It is worth noting that, unlike Dirichlet problems for BSPDEs (see \cite{QiuTangMPBSPDE11}) or Neumann problems for deterministic PDEs (see \cite{Lieberman}), It\^o's formula for the square norm is not well-defined for the weak solutions of the Neumann problems for BSPDEs with a nontrivial coefficient $\sigma$  and this makes the existing methods for weak solutions inapplicable here (Remark \ref{rmk-linear}).} Then the existence and uniqueness of strong solution to \eqref{BSPDE-NP} follows immediately. However, to verify that the obtained solution is the value function and to derive the optimal feedback control for problem  \eqref{min-contrl-probm}-\eqref{state-proces-contrl}, we need to make sense of the composition of the solution of \eqref{BSPDE-NP} and the controlled state process $X$, and this requires improved regularity of $u$. Inspired by the smoothing properties of the leading operators of BSPDEs (see \cite{Qiu-2014-Hormander}), we assume that  $\bar{\sigma}$ satisfies the {\sl super-parabolicity condition:} 
   \begin{align*}
     &  (\mathcal{A} 1)\quad \text{There exists constant }\kappa, \text{ s.t. } \left| \bar\sigma_t(x) \right|^2 \geq \kappa>0 \quad \text{a.s.,}  \\ & \,\,\forall\, (t,x)\in [0,T]\times\bR.
   \end{align*}
  By $(\cA 0)$, the randomness of all the coefficients for each fixed time and state is only subject to the sub-filtration $\{\sF_t\}_{t\geq 0}$ generated by Wiener process $W$, which allows our set-up to go beyond the classical Markovian  framework and furthermore, together with $(\cA 1)$, ensures the super-parabolicity and thus smoothing property of the involved differential operator in BSPDE \eqref{BSPDE-NP}; refer to \cite{Qiu-2014-Hormander} for more detailed discussions\footnote{In fact, according to the investigations of \cite{Qiu-2014-Hormander}, the randomness subject to sub-filtration $\{\sF_t\}_{t\geq 0}$ may damage the regularity of solutions, while the terms associated with Wiener process $B$, seen as the Markovian part, serve to restore the smoothing property. A sufficiently regular solution is needed for the verification theorem as well as for the construction of the optimal control. Therefore, we introduce two independent Wiener processes $W$ and $B$ and assume the super-parabolicity $(\mathcal{A} 1)$.}.  Then, we take spatial derivatives on both sides of \eqref{BSPDE-NP}. The resulting Dirichlet problem admits a unique strong solution (see \cite{DuTang2010}), which yields additional regularity of $Du$. Finally, the generalized It\^o-Kunita-Wentzell formula, applicable to the sufficiently regular random field $u_t(x)$, allows us to finish the verification.

The \emph{nonlinear} BSPDE like \eqref{BSPDE-NP} is called stochastic Hamilton-Jacobi-Bellman (HJB) equation, which was first introduced by Peng \cite{Peng_92} for controlled SDEs without reflection. For the utility maximization with habit formation, a specific fully nonlinear stochastic HJB equation was formulated by Englezos and Karatzas \cite{EnglezosKaratzas09} and the value function was verified to be its classical solution. The study of \emph{linear} BSPDEs, on the other hand,  dates back to  about thirty years ago (see Bensoussan
\cite{Bensousan_83} and Pardoux \cite{Pardoux1979}). They arise in many applications of probability theory and stochastic processes, for
instance in the nonlinear filtering and stochastic control theory for processes with incomplete information, as an adjoint equation of the Duncan-Mortensen-Zakai filtering equation (for instance, see \cite{Bensousan_83,Hu_Ma_Yong02,Hu_Peng_91,Zhou_93}). The representation relationship between forward-backward stochastic differential equations and BSPDEs yields the stochastic Feynman-Kac formula (see \cite{Hu_Ma_Yong02,ma1999linear,QiuTangYou-SPA-2011}). In addition, as the obstacle problems of BSPDEs, the reflected BSPDE arises as the HJB equation for the optimal stopping problems (see \cite{ChangPangYong-2009,Oksend-Sulem-Zhang-2011,QiuWei-RBSPDE-2013}).

The linear and semilinear BSPDEs have been extensively studied, we refer to \cite{DuTang2010,Hu_Ma_Yong02,ma2012non,ma1999linear,Tang-Wei-2013} among many others. For the weak solutions and associated local behavior analysis for general quasi-linear BSPDEs, see \cite{QiuTangBDSDES2010,QiuTangMPBSPDE11}, and we refer to \cite{GraeweHorstQui13} for BSPDEs with singular terminal conditions. However,  the existing literature is mainly about the BSPDEs in the whole space and Dirichlet problem, and not on the Neumann problem, though some partial results could be concluded from the semigroup method of BSPDEs \cite{Hu_Peng_91,Tessitore_96} for the cases when $\sigma\equiv 0$.

The remainder of this paper is organized as follows. In section 2, we summarize the main assumptions and results. The existence and uniqueness of strong solution for the Neumann problem of general nonlinear BSPDEs is established in Section 3, where we first give the a priori estimates of strong solutions for linear equations and then use the continuity method to prove the well-posedness for the general nonlinear cases. In Section 4,  we complete the proof of the main theorem. Finally, the appendix recalls an It\^o formula for the square norms of solutions of SPDEs and provides the sketched proof for a generalized It\^o-Kunita-Wentzell formula.

\section{Preliminaries and Main Result}


\subsection{Notations and definition of solutions to BSPDEs}

In this paper, we use the following notation. $D$ and $D^2$ {denote} the first {order} and second order spatial partial derivative {operators}, respectively; the other partial derivatives are denoted by $\partial$.
For a Banach space $V$, the space $L^2(\Omega,\sF_T;V)$ is the set of all $V$-valued $\sF_T$-measurable and square-integrable random variables, and   we denote by $\cS^p_{\sF} ([0,T];V)$, $p\in[1,\infty)$, the set of all the $V$-valued and $\sP$-measurable c\`adl\`ag processes $(X_{t})_{t\in [0,T]}$ such that
\[
	\|X\|_{\cS _{\sF}^p([0,T];V)}^p= E \sup_{t\in [0,T]} \|X_t\|_V^p< \infty.
\]
By $\cL^p_{\sF}(0,T;V)$ we denote the class of $V$-valued $\sP$-measurable processes $(u_t)_{t\in[0,T]}$ such that
\begin{align*}
\|u\|^p_{\cL^p_{\sF}(0,T;V)} &=E \int_0^T\|u_t\|^p_{V}\,dt<\infty, \quad p\in [1,\infty);\\
 \|u\|_{\cL^{\infty}_{\sF}(0,T;V)}&= \esssup_{(\omega,t)\in\Omega\times[0,T]}\|u_t\|_{V}<\infty, \quad p=\infty.
\end{align*}
In a similarly way, we define $\cS _{\bar{\sF}}^p([0,T];V)$ and $\cL^p_{\bar{\sF}}(0,T;V)$. For the two spaces $\cS_{\sF}^2([0,T];V)$ and $\cL_{\sF}^2(0,T;V)$, we omit the subscript for simplicity, especially when there is no confusion on the filtration and adaptedness.

For $k \in \mathbb{N}^+$ and $p\in [1,\infty)$, $H^{k,p}([0,b])$ is the Sobolev space of all real-valued functions $\phi$ whose up-to $k$th order derivatives belong to $L^p([0,b])$, equipped with the usual Sobolev norm $\|\phi\|_{H^{k,p}([0,b])}$. By $H^{k,p}_0([0,b])$, we denote the space of all the trace-zero functions in $H^{k,p}([0,b])$.
For $k=0$, $H^{0,p}([0,b])\triangleq L^p([0,b])$. For simplicity, by $u=(u_1,\dots,u_l)\in H^{k,p}([0,b])$, we mean $u_1,\dots, u_l\in H^{k,p}([0,b])$  and $\|u\|^p_{H^{k,p}([0,b])}=\sum_{j=1}^l \|u_j\|^p_{H^{k,p}([0,b])}$. We use $\|\cdot\|$ and $\langle \cdot,\,\cdot\rangle$ to denote the norm and the inner product in the usual Hilbert spaces $L^2([0,b])$, and if there is no confusion, we shall also use $\langle\cdot,\, \cdot \rangle$ to denote the duality between Hilbert space $H^{k,2}([0,b])$ and their dual spaces.

Throughout this paper, we shall use $C$ to denote a constant whose value may vary from line to line and we set for $k=1,\,2$
\begin{align*}
\cH=& \cS^2_{\sF}(0,T;L^2([0,b]))\cap  \cL^2_{\sF}(0,T;H^{1,2}([0,b])) \times  \cL^2_{\sF}(0,T;L^2([0,b])),\\
\cH^k=& \cS^2_{\sF}(0,T;H^{k,2}([0,b]))\cap  \cL^2_{\sF}(0,T;H^{k+1,2}([0,b])) \times  \cL^2_{\sF}(0,T;H^{k,2}([0,b])),
\end{align*}
and they are complete spaces equipped respectively with the norms
\begin{align*}
&\|(u,\psi)\|^2_{\cH}\\
&= \|u\|^2_{\cS^2_{\sF}(0,T;L^2([0,b]))} + \|u\|^2_{ \cL^2_{\sF}(0,T;H^{1,2}([0,b]))} +\|\psi\|^2_{\cL^2_{\sF}(0,T;L^2([0,b]))}, \\& \quad\quad \text{for }(u,\psi)\in\cH,\\
&\|(u,\psi)\|^2_{\cH^k}
\\
&= \|u\|^2_{\cS^2_{\sF}(0,T;H^{k,2}([0,b]))} + \|u\|^2_{ \cL^2_{\sF}(0,T;H^{k+1,2}([0,b]))} 
 +\|\psi\|^2_{\cL^2_{\sF}(0,T;H^{k,2}([0,b]))},\\&  \quad \quad \text{for }(u,\psi)\in\cH^k.
\end{align*}

%
Finally, we introduce the notion of solutions to BSPDEs with general nonlinear coefficients which are not restricted to the forms of BSPDE \eqref{BSPDE-NP}. 

\begin{defn}\label{defn-solution}
Let $G\in L^2(\Omega,\sF_T;L^2([0,b]))$ and  $F$ be a random function such that for any $x,x^1,x^2\in\bR$ and any $z,z^1\in\bR^m$
$$
F_{\cdot}(\cdot,x,x^1,x^2,z,z^1):\ \Omega\times[0,T]\times[0,b]\ \rightarrow \bR
$$
 is $\sP\otimes\mathcal{B}([0,b])$-measurable.
 A pair of processes $(u,\psi)$ is a {\sl weak} solution to the BSPDE
  \begin{equation}\label{BSPDE-defn}
  \left\{\begin{array}{l}
  \begin{aligned}
  -du_t(y)=\, &F_t(y,u,Du,D^2u,\psi,D \psi)\, dt
  -\psi_t(y)\, dW_{t}, \\
  &\quad (t,y)\in [0,T]\times [0,b];\\
  Du_t(0)=\,&g_t(0), \quad Du(b)=g_t(b);   \\
    u_T(y)=\,& G(y),  \quad y\in[0,b],
    \end{aligned}
  \end{array}\right.
\end{equation}
if $(u,\psi)\in \cH$  with the traces of $Du(t,\cdot)$ coinciding with $g_t(0)$ and $g_t(b)$ at the boundary,
and $(u,\psi)$ satisfies BSPDE \eqref{BSPDE-defn} in the weak sense, i.e., for any $\varphi\in C_c^{\infty}((0,b))$, 
$$\left\langle \varphi,\,F_{\cdot}(\cdot,u,Du,D^2u,\psi,D \psi)\right\rangle \in \cL^1_{\sF}(0,T;\bR),$$ and
  \begin{align}\label{eq-defn-2}
    \langle \varphi,\, u_t\rangle
    &= \langle\varphi,\, G\rangle+\!
    \int_t^T\!\!\!\langle \varphi,\, F_s(\cdot,u,Du,D^2u,\psi,D \psi)\rangle\,ds \nonumber \\
    &\quad\quad\quad -\!\int_t^T\!\!\langle\varphi,\,\psi_sdW_s\rangle\, \, {\text{a.s.},} \quad t \in [0,T].
  \end{align}
  The above $(u,\psi)$ is called a {\sl strong} solution if we have improved regularity $(u,\psi)\in \cH^1 $.
\end{defn}

It is easy to see that in BSPDE \eqref{BSPDE-NP} we have a particular case of nonlinear term $F$ with 
\[
\begin{split}
&F_t(y,u,Du,D^2u,\psi,D\psi)
\\
&=  \frac{1}{2}\left( | \sigma_t(y)|^2+|\bar\sigma_t(y)|^2\right)D^2 u_t(y)
         + \sigma_t(y) D \psi_t(y)+{\mathbb H} _t(y,Du_t(y)).
      \end{split}  
        \]

\subsection{Assumptions and main result}

For the well-posedness of BSPDE \eqref{BSPDE-NP}, we use the following assumptions. 
\begin{enumerate}[$({\mathcal A} 2)$]
	\item  
The functions $\sigma$, $\bar{\sigma}$ and their spatial partial derivatives $D\sigma$, $D\bar{\sigma}$
are $\sP\otimes \mathscr{B}(\bR)$-measurable and essentially bounded by a positive constant $K>0$. And the functions $\beta$, $f$ and the  spatial partial derivative $D\beta$ are $\sP\otimes \mathscr{B}(\bR) \otimes \mathscr{B}(\bR^n)$-measurable with $\beta_{\cdot}(0,\theta)\in \cL^2_{\sF}(0,T;\bR)$ and $\left|D\beta_t(x,\theta)\right|\leq \Lambda$ a.s. for any $(t,x,\theta)\in [0,T]\times\bR\times\bR^n$.
\end{enumerate}
\begin{enumerate}[$({\mathcal A^*}) $]
	  	\item (i)
	  $$G\in L^2(\Omega,\sF_T; H^{2,2}([0,b])), \quad DG-g_T\in  L^2(\Omega,\sF_T; H_0^{1,2}([0,b])), $$ and together with another function $\cG $, the pair $(g,\cG)$ belongs to $\cH^1$ and satisfies BSPDE $-dg_t=\mathscr{G}_t\,dt-\cG_tdW_t$ in the weak sense (see Definition \ref{defn-solution}) with $ \mathscr G\in\cL^2_{\sF}(0,T;L^2([0,b])) $.\\[4pt]
	(ii) For any $v\in \cS^2_{\sF}(0,T;H^{1,2}([0,b]))\cap  \cL^2_{\sF}(0,T;H^{2,2}([0,b])) $ we have that ${\mathbb H} _{\cdot}(\cdot,v), (D{\mathbb H} )_{\cdot}(\cdot,v) \in \cL^2_{\sF}(0,T;L^2([0,b]))$, and there exists a \textit{nonnegative} constant $K_0$ such that for any $v_1,v_2\in\bR$, there holds almost surely
	\begin{align*}
	&\left|{\mathbb H} _t(x, v_1)  - {\mathbb H} _t(x, v_2)  \right| 
	 \leq 
	 K_0  \left|    v_1-v_2  \right|      ,\quad \text{for all }(t,\,x)\in [0,T]\times[0.b].
	\end{align*} \\[2pt]
	(iii) There exists a $\sP\otimes \mathscr{B}(\bR)\otimes\mathscr{B}(\bR)$-measurable $\Theta$-valued function $\Pi$ such that ${\mathbb H} _t(x,y)=\beta_t(x,\Pi_t(x,y))y+f_t(x,\Pi_t(x,y))$, i.e.,
	$$
	\Pi_t(x,y)\in \text{arg}\essinf_{\theta\in \Theta} \{  \beta_t(x,\theta) y + f_t(x,\theta)  \},
	$$
	 and for each $v\in \cS^2_{\sF}(0,T;H^{1,2}([0,b]))\cap  \cL^2_{\sF}(0,T;H^{2,2}([0,b])) $, the reflected SDE \eqref{state-proces-contrl} associated with drift coefficient $\beta_t(X_t,\Pi_t(X_t,v_t))$ has a unique solution.
\end{enumerate}

Now, we state the main theorem, whose proof requires some preparations which will be carried out subsequently.

\begin{thm}\label{thm-verification}
Let assumptions $(\cA 0)$, $(\cA 1)$, $(\cA 2)$ and $(\cA^*)$ hold {with $\sigma_t(0)=\sigma_t(b)=0$ a.s. for any $t\in[0,T]$}. BSPDE  \eqref{BSPDE-NP} admits a unique strong solution $(u,\psi)$. For this strong solution, we have further $(u,\psi)\in \cH^2$. Moreover, $u$ turns out to be the value function of the stochastic control problem \eqref{min-contrl-probm}, and the optimal control $\theta^*$ and state process $X^*$ are given by $\theta^*=\Pi_t(X^*_t,Du_t(X^*_t))$ and 
\begin{equation}\label{state-proces-contrl-optimal}
\left\{
\begin{split}
&dX^*_t=  \beta_t(X^*_t,\Pi_t(X^*_t,Du_t(X^*_t)))\,dt+\sigma_t(X^*_t)\,dW_t+\bar\sigma_t(X^*_t)\,dB_t \\
&\quad\quad\quad\quad +dL_t-dU_t; \\
&X^*_0=x; \quad L_0=U_0=0;\\
&0\leq X^*_t\leq b, \quad \text{a.s.};\\
&\int_0^T\!\! X^*_t\,dL_t =\int_0^T(b-X^*_s)\,dU_s=0,\quad \text{a.s.}
\end{split}
\right.
\end{equation}
\end{thm}

The conditions $(\cA 0)$, $(\cA 1)$ and $(\cA 2)$ are considered as standing assumptions throughout this paper and they are standard to guarantee the adaptedness and super-parabolicity of BSPDE \eqref{BSPDE-NP} and the well-posedness of the controlled reflected SDEs (see \cite[Theorem 3.1 and Remark 3.3]{lions1984stochastic}). 

In assumption (i) of $(\cA ^*)$, to have $(u,\psi)\in\cH^2$, the requirements on $G$ is standard (see $L^p$-theory of BSPDE of \cite{DuQiuTang10}); in view of the Skorohod conditions of RSDE \eqref{state-proces-contrl}, one has
$$\int_0^Tg_s(X_s)\, dL_s=\int_0^Tg_s(0)\, dL_s,\quad \text{and}\quad \int_0^T g_s(X_s)\, dU_s=\int_0^Tg_s(b)\, dU_s,$$ 
so only the traces $g_s(0)$ and $g_s(b)$ of $g$ are involved in the control problem. In fact, assumption (i) of $(\cA ^*)$  allows $g_s(0)$ and $g_s(b)$ to be any processes that, together with another two processes $(\zeta^0,\zeta^b)$, satisfy BSDEs of the following form:
\begin{align*}
g_t(0)&=DG(0)+\int_t^T\tilde{g}^0_s\,ds-\int_t^T\zeta^0_s\,dW_s;\\
g_t(b)&=DG(b)+\int_t^T\tilde{g}^b_s\,ds-\int_t^T\zeta^b_s\,dW_s,
\end{align*}
with $\tilde{g}^0,\tilde{g}^b\in\cL^2(0,T;\bR)$, and we can construct (not uniquely) the time-space random function $g_t(x)$ in different ways. For instance, starting with $(g_t(0),g_t(b))$, one can construct linearly
$$
g_t(x)=g_t(0)+\frac{\left(g_t(b)-g_t(0)\right)x}{b},\quad \quad (t,x)\in[0,T]\times[0,b],
$$
which then satisfies assumption (i) of $(\cA ^*)$ with 
$$
\mathscr G_t(x)=\tilde{g}_t^0+\frac{\left(\tilde g_t^b-\tilde g_t^0\right)x}{b} \text{ and } \cG_t(x)=\zeta_t^0+\frac{\left(\zeta_t^b-\zeta_t^0\right)x}{b} \quad (t,x)\in[0,T]\times[0,b].
$$
In this paper, we adopt assumption (i) of $(\cA ^*)$  for the  convenience of discussions.

By (ii) of $(\cA ^*)$, we assume the Lipchitz continuity of Hamiltonian function ${\mathbb H} _t(x,v)$ with respect to $v$, which implies $\partial_v{\mathbb H} _{\cdot}(\cdot,v) \in L^{\infty}(\Omega\times[0,T]\times [0,b])$ for any $v\in\bR^d$. This excludes the control problems of linear-quadratic type. The quadratic case definitely needs more efforts, for which we need to deal with not only the quadratic growth but also the improved regularity in Theorem \ref{thm-verification}, so we would postpone the discussions on quadratic cases to a future work.

In (iii) of $(\cA ^*)$, $\Pi$ is the minimizer function of ${\mathbb H} _t(x,v)$ (see \eqref{eq-hamiltonian}) and for each $u\in  \cS^2_{\sF}(0,T;H^{2,2}([0,b]))\cap  \cL^2_{\sF}(0,T;H^{3,2}([0,b])) $, the composite function $\beta_t(x,\pi_t(x,Du_t(x)))$ may not be Lipchitz continuous with respect to $x$. The following example contains such a case. But still the reflected SDE \eqref{state-proces-contrl} has a unique strong solution.


\begin{exam}
Let $d=n=1$, $\Theta=[-1,0]$ and $\beta_t(x,\theta)\equiv \theta$, while $f_t(x,\theta)=\mu |\theta| + h_t(x)$ with $\mu\in\bR^+$ and $h\in\cL^2_{\sF}(H^1([0,b]))$. 
Suppose $\sigma$ and $\bar\sigma$ satisfy the assumptions in Theorem \ref{thm-verification}.
Assume $g_t(x)\equiv \frac{px}{b}$ as in \cite{ata2005drift}. Let $G(x)=\frac{px^2}{2b}$.
Then
$$
{\mathbb H} _t(x,Du_t)=\essinf_{-1\leq \theta \leq 0}
\left\{
\theta Du_t(x) +\mu |\theta |  +h_t(x)
\right\}  
=-\left( Du_t(x)-\mu  \right)^+ +h_t(x),
$$
and 
$$
\Pi_t(x,Du_t(x))= -\textbf{1}_{\{Du_t(x)>\mu\}}.
$$
It is easy to check that $(\cA 0)-(\cA 2)$ and (i) and (ii) of $(\cA^*)$ hold. Obviously,  the drift $\beta=\Pi$, as a step function, is not necessarily Lipchitz continuous with respect to $x$ for each $u\in \cS^2_{\sF}(0,T;H^{2,2}([0,b]))\cap  \cL^2_{\sF}(0,T;H^{3,2}([0,b])) $. In our case, the resulting reflected SDE reads
\begin{equation}\label{state-proces-contr-examl}
\left\{
\begin{split}
&dX_t= -\textbf{1}_{\{Du_t(X_t)>\mu\}} \,dt+\sigma_t(X_t)\,dW_t+\bar\sigma_t(X_t)\,dB_t+dL_t-dU_t,   \\
&\quad\quad\quad \quad  t\in[0,T] ;\\
&X_0=x; \quad L_0=U_0=0;\\
&0\leq X_t\leq b, \quad \text{a.s.};\\
&\int_0^T\!\! X_t\,dL_t =\int_0^T(b-X_s)\,dU_s=0,\quad \text{a.s.}
\end{split}
\right.
\end{equation}
 In fact, given $u\in \cS^2_{\sF}(0,T;H^{2,2}([0,b]))\cap  \cL^2_{\sF}(0,T;H^{3,2}([0,b])) $, $X$ is the unique solution  to the reflected SDE
\begin{equation}\label{state-proces-contr-examl-0drift}
\left\{
\begin{split}
&dX^0_t= \sigma_t\,dW_t+\bar\sigma_t\,dB^{\mathbb Q}_t+dL_t-dU_t,    \quad  t\in[0,T] ;\\
&X^0_0=x; \quad L_0=U_0=0;\\
&0\leq X^0_t\leq b, \quad \text{a.s.};\\
&\int_0^T\!\! X^0_t\,dL_t =\int_0^T(b-X^0_s)\,dU_s=0,\quad \text{a.s.,}
\end{split}
\right.
\end{equation}
where $B^{\mathbb Q}$ is a Wiener process under the equivalent probability measure $\mathbb{Q}$ with
  \begin{equation*}
  \begin{split}
  \frac{d\mathbb{Q}}{d\mathbb{P}} :=&\exp\biggl( \int_0^T  \textbf{1}_{\{Du_t(X_t)>\mu\}}  \left |\bar\sigma_t (X_t)\right|^{-1}\,dB_s
  \\
  &
  \quad\quad\quad\quad\quad\quad\quad
  -\frac{1}{2}\int_{0}^T\textbf{1}_{\{Du_t(X_t)>\mu\}}  \left |\bar\sigma_t(X_t)\right|^{-2} \,ds\biggr),
  \end{split}
\end{equation*}
and analogous to \cite[Propositions 3.6 \& 3.10 ]{karatzas2012brownian}, Girsanov theorem implies the unique existence of the \textit{weak} solution for reflected SDE \eqref{state-proces-contr-examl}. In order to get the unique existence of (strong) solution, by analogy to \cite[Corollary 3.23]{karatzas2012brownian},  it remains to prove the uniqueness of (strong) solution (also called pathwise uniqueness) to  reflected SDE \eqref{state-proces-contr-examl}. Suppose $(X^1,L^1,U^1)$ and $(X^2,L^2,U^2)$ are two (strong) solutions of \eqref{state-proces-contr-examl} (on the same probability space). Simple calculations give
\begin{align*}
&(X^1_t-X^2_t)^+\\
&= \int_0^t \left( \textbf{1}_{\{Du_s(X^2_s)>\mu\}} -\textbf{1}_{\{Du_s(X^1_s)>\mu\}} \right) \textbf{1}_{\{ X^1_s>X^2_s \}}\,ds 
		\\
		&\quad \quad 
		+\int_0^t \textbf{1}_{\{ X^1_s>X^2_s \}} (dL^1_s-dL^2_s)
		-\int_0^t \textbf{1}_{\{ X^1_s>X^2_s \}} (dU^1_s-dU^2_s)\\
		&\quad \quad 
		+\int_0^t \textbf{1}_{\{ X^1_s>X^2_s \}} (\sigma_s(X^1_s)-\sigma_s(X^2_s))\,dW_s\\
		&\quad\quad
		+\int_0^t \textbf{1}_{\{ X^1_s>X^2_s \}} (\bar\sigma_s(X^1_s)-\bar\sigma_s(X^2_s))\,dB_s
		\\
&= \int_0^t \left( \textbf{1}_{\{Du_s(X^2_s)>\mu\}} -\textbf{1}_{\{Du_s(X^1_s)>\mu\}} \right) \textbf{1}_{\{ X^1_s>X^2_s \}}\,ds -\int_0^t \textbf{1}_{\{ X^1_s>0 \}}\,dL^2_s
\\
&\quad
\quad-\int_0^t \textbf{1}_{\{ X^2_s<b \}}\,dU^1_s
	+\int_0^t \textbf{1}_{\{ X^1_s>X^2_s \}} (\sigma_s(X^1_s)-\sigma_s(X^2_s))\,dW_s \\
&\quad\quad
		+\int_0^t \textbf{1}_{\{ X^1_s>X^2_s \}} (\bar\sigma_s(X^1_s)-\bar\sigma_s(X^2_s))\,dB_s,
\end{align*}
where Skorohod conditions indicate relations
{\small{$$
\int_0^t \textbf{1}_{\{ X^1_s>X^2_s \}} dL^1_s= \int_0^t \textbf{1}_{\{ 0>X^2_s \}} dL^1_s=0=\int_0^t \textbf{1}_{\{ X^1_s>b \}}  dU^2_s=\int_0^t \textbf{1}_{\{ X^1_s>X^2_s \}}  dU^2_s.
$$
}}
Therefore,
\begin{align*}
&X^1_t \vee  X^2_t
\\
&= X^2_t+(X^1_t-X^2_t)^+  \\
&= x-\int_0^t \textbf{1}_{\{Du_s(X^2_s)>\mu\}}  \,ds \\
&\quad+ \int_0^t \left( \textbf{1}_{\{Du_s(X^2_s)>\mu\}} -\textbf{1}_{\{Du_s(X^1_s)>\mu\}} \right) \textbf{1}_{\{ X^1_s>X^2_s \}}\,ds\\
&\quad + L^2_t-\int_0^t \textbf{1}_{\{ X^1_s>0 \}}\,dL^2_s
-\int_0^t \textbf{1}_{\{ X^2_s<b \}}\,dU^1_s- U^2_t
\\
&\quad
+\int_0^t\left( \sigma_s(   X^2_s)+ \textbf{1}_{\{ X^1_s>X^2_s \}} (\sigma_s(X^1_s)-\sigma_s(X^2_s))  \right)\,dW_s\\
&\quad
		+\int_0^t\left( \bar\sigma_s(  X^2_s)+\textbf{1}_{\{ X^1_s>X^2_s \}} (\bar\sigma_s(X^1_s)-\bar\sigma_s(X^2_s)) \right)\,dB_s\\
&= x-\int_0^t \textbf{1}_{\{Du_s(X^1_s\vee X^2_s)>\mu\}}  \,ds  +\int_0^t \,d\check L_s
-\int_0^t \,d\check U_s\\
&\quad
+\int_0^t\sigma_s(X^1_s \vee  X^2_s)\,dW_s
		+\int_0^t\bar\sigma_s(X^1_s \vee  X^2_s)\,dB_s,
\end{align*}
with $d\check L_s= \textbf{1}_{\{ X^1_s\leq 0 \}}\,dL^2_s$ and $d\check U_s= \textbf{1}_{\{ X^2_s<b \}}\,dU^1_s+ dU^2_s$.
Noticing 
$$0\leq \left( X^1_t\vee X^2_t \right) \textbf{1}_{\{ X^1_t\leq 0 \}}\,dL^2_t 
\leq  X^2_t  dL^2_t =0, $$
and 
\begin{align*}
0&\leq \left(b-X^1_t\vee X^2_t \right)  \left(  \textbf{1}_{\{ X^2_t<b \}}\,dU^1_t +d U^2_t   \right)
\\
&\leq \left(b-X^1_t\right)   \,dU^1_t +\left(b- X^2_t \right) dU^2_t =0,
\end{align*}
we see that $(X^1\vee X^2,\check L,\check U)$ is also a (strong) solution. Hence, $X^1$ and $X^1\vee X^2$ have the same probability law and this is only true if $X^1$ and $X^1\vee X^2$ are indistinguishable, i.e. the pathwise uniqueness holds. This finally indicates that reflected SDE \eqref{state-proces-contr-examl} has a unique (strong) solution. Therefore, the assumption (iii) of $(\cA^*)$ is satisfied and Theorem \ref{thm-verification} applies.

\end{exam}


\section{Existence and uniqueness of a strong solution for general nonlinear BSPDEs}

In this section, we shall establish the existence and uniqueness of strong solution for the Neumann problem for general nonlinear BSPDEs, which might be of interest even out of the current context. For simplicity, we only consider the 1-dimensional case, though there would be no essential difficulty for multi-dimensional extensions.

Consider the following Neumann problem:
\begin{equation}\label{BSPDE-NP-nonlinear}
  \left\{\begin{array}{l}
  \begin{aligned}
  -du_t(x)=\, &\bigg[ \frac{1}{2}\left( | \sigma_t(x)|^2+|\bar\sigma_t(x)|^2\right)D^2 u_t(x)
        + \sigma_t D \psi_t(x)\\&+\Gamma_t(x,u,Du,D^2u,\psi,D\psi)
         \bigg]\, dt  
        -\psi_t(x)\, dW_{t}, \\
        Du_t(0)=\,&0,\quad Du_t(b)=0;\\
    u_T(x)=\,& G(x),  \quad x\in[0,b].
    \end{aligned}
  \end{array}\right.
\end{equation}

The following assumption is restricted to this section.
\begin{enumerate}[$({\mathcal A} 3)$]
	\item   For any $(u,\psi)\in H^{2,2}([0,b])\times H^{1,2}([0,b])$, $\Gamma_{\cdot}(\cdot,u,Du,D^2u,\psi,D\psi)\in\cL^2(0,T;L^2([0,b]))$, and there exist \textit{nonnegative} constants $\mu$ and $L$ such that for any $(u_i,\,\psi_i)\in H^{2,2}([0,b])\times H^{1,2}([0,b])$, $i=1,2$, there holds
	\begin{align*}
	&\left\|\Gamma_t(\cdot, u_1,Du_1,D^2u_1,\psi_1,D\psi_1 )  - \Gamma_t(\cdot, u_2,Du_2,D^2u_2,\psi_2,D\psi_2 )  \right\| \\
	& \leq 
	\mu \left( \left\|D^2(u_1-u_2)\right\| + \left\|D(\psi_1-\psi_2) \right\|\right) 
	\\& + L \left( \left\|    u_1-u_2  \right\|_{H^{1,2}([0,b])}      +   \left\|    \psi_1-\psi_2  \right\|_{L^{2}([0,b])}  \right),\quad \text{a.s.,}
	\end{align*} 
	for any $t\in[0,T]$.
\end{enumerate}

\begin{rmk} \label{rmk-fully-nonlinear}
Assumption $(\cA 3)$ holds for the following semi-linear fucntional:
$$
\Gamma_t(x,u,Du,D^2u,\psi,D\psi)
=\alpha_tDu_t(x) + c_tu_t(x)+\gamma_t\psi+h_t(x,u,Du,\psi)
$$
with bounded coefficients $\alpha$, $c$, $\gamma$ and a certain Lipchitz continuous (w.r.t. $(u,Du,\psi)$)  function $h$. In particular, letting assumptions $(\cA 0)-(\cA 2)$ and $(\cA^*)$ hold, the Hamiltonian function $\mathbb H_t(x, Du_t(x))$ in BSPDE \eqref{BSPDE-NP} satisfies assumption $(\cA 3)$. More examples can be constructed in a similar way to \cite[Remark 5.1]{DuQiuTang10}. It is worth noting that Assumption $(\cA 3)$ allows $\Gamma$ to be fully nonlinear with a small dependence on $D^2u$ and $D\psi$.
\end{rmk}

The existence and uniqueness of strong solution to BSPDE \eqref{BSPDE-NP-nonlinear} is summarized below.
\begin{thm}\label{thm-Neumann-nonlinear}
Let $G\in L^2(\Omega,\sF_T;H^{1,2}([0,b]))$ and assumptions $(\cA0)$, $(\cA1)$, $(\cA2)$ and $(\cA3)$ hold. There exists a positive constant $\mu_0$ depending on $\kappa$, $L$, $K$ and $T$, such that when $0\leq \mu<\mu_0$, BSPDE \eqref{BSPDE-NP-nonlinear}  admits a unique strong solution $(u,\psi)$ satisfying
\begin{align}
\|(u,\psi)\|_{\cH^1}
\leq C\left(
\|G\|_{L^2(\Omega,\sF_T;H^{1,2}([0,b]))} 
+\left\|\Gamma^0 \right\|_{\cL^2(0,T;L^2([0,b]))}
\right), \label{thm-est}
\end{align}
where $\Gamma^0\triangleq \Gamma_{\cdot}(\cdot,0,0,0,0,0)$ and the constant $C$ depends on $\mu$, $L$, $\kappa$, $K$ and $T$.
\end{thm}

For the proof of Theorem \ref{thm-Neumann-nonlinear}, we shall first establish a priori estimates for some linear equations  in Section \ref{subsec:estimates} and then use the method of continuity to complete the proof in Section \ref{subsec:general nonlinear}. The readers may turn to Section \ref{subsec:proof} for the proof of Theorem \ref{thm-verification} for the main result of this work.

\subsection{The a priori estimates} \label{subsec:estimates}

For each $\lambda\in[0,1]$, we consider the following linear BSPDE:
\begin{equation}\label{BSPDE-NP-linear}
  \left\{\begin{array}{l}
  \begin{aligned}
  -du_t(x)=\, &\bigg[ \frac{\lambda}{2}\left( | \sigma_t(x)|^2+|\bar\sigma_t(x)|^2\right)D^2 u_t(x)
        +\lambda \sigma_t (x)D \psi_t(x)\\
        &+\frac{1-\lambda}{2} D^2u_t(x)
        +h_t(x)
         \bigg]\, dt  
        -\psi_t(x)\, dW_{t}, \\
        &\quad\quad \quad (t,x)\in [0,T]\times [0,b];\\
        Du_t(0)=\,&0,\quad Du_t(b)=0;\\
    u_T(x)=\,& G(x),  \quad x\in[0,b].
    \end{aligned}
  \end{array}\right.
\end{equation}

\begin{prop}\label{prop-apriori}
Let $(\cA 0)$, $(\cA 1)$ and $(\cA 2)$ hold and 
$$h\in\cL^2(0,T;L^2([0,b])),\quad G\in L^2(\Omega,\sF_T;H^{1,2}([0,b])).$$ Suppose $(u,\psi)$ is a strong solution of Neumann problem \eqref{BSPDE-NP-linear}. Then the strong solution is unique and it satisfies
\begin{align*}
&\|(u,\psi)\|^2_{\cH}
\\
&\leq
C_1\, \left\{ 
\|G\|^2 _{L^2(\Omega,\sF_T;L^2([0,b]))}
+ E\Bigg[ \int_0^T   
\left|\llangle h_s,\, u_s  \rrangle \right| 
+\frac{1}{\eps} \|u_s\|^2
+\eps   \|D\psi_s\|^2
\,ds \Bigg]\right\}, \\
& \forall \eps>0, \\
&\text{and }\\
&\|(u,\psi)\|^2_{\cH^1}
\\
& \leq C_2\,\left\{
\|G\|^2_{L^2(\Omega,\sF_T;H^{1,2}([0,b]))} + E\left[ \int_0^T \left(\left| \llangle h_t, \, u_t\rrangle\right| 
+\left| \llangle h_t,\,D^2u_t\rrangle \right| \right) \,dt \right]
\right\}\\
&\leq 
C_3\,\left\{
\|G\|^2_{L^2(\Omega,\sF_T;H^{1,2}([0,b]))} + \|h\|^2_{\cL^2_{\sF}(0,T;L^2([0,b]))}
\right\},
\end{align*}
where the constants $C_1,\,C_2$ and $C_3$ depend only on $\kappa$, $K$ and $T$ and are independent of $\lambda\in[0,1]$.
\end{prop}

\begin{proof}
\tbf{Step 1.} Applying It\^o's formula (see Lemma \ref{lem-ito-formula}) to the square norm yields
\begin{align*}
&\|u_t\|^2+\int_t^T\|\psi_s\|^2\, ds-\|G\|^2\\
&= \int_t^T \llangle u_s,\,
\lambda \left[ (|\sigma_s|^2+|\bar\sigma_s|^2)D^2u_s 
+ 2  \sigma_sD\psi_s \right] + (1-\lambda)D^2u_s
+2 h_s
    \rrangle\,ds
\\& -2\int_t^T\llangle u_s,\, \psi_s\,dW_s\rrangle,\quad \text{a.s. }
\end{align*}
for any $ t\in[0,T]$. In view of the Neumann boundary condition, we have 
$$\int_t^T\llangle u_s,\,(1-\lambda) D^2u_s \rrangle = -(1-\lambda)\int_t^T \|Du_s\|^2\,ds$$
 and 
\begin{align*}
& \int_t^T \llangle u_s,\,(|\sigma_s|^2+|\bar\sigma_s|^2)D^2u_s\rrangle\,ds
\\& =- \int_t^T \!\!\!   \llangle Du_s,\,(|\sigma_s|^2+|\bar\sigma_s|^2)Du_s\rrangle\,ds -\int_t^T\!\!\!  \llangle u_s D(|\sigma_s|^2+|\bar\sigma_s|^2) ,\,Du_s\rrangle\,ds
\\
&\text{(by ($\cA 2$))}\\
&\leq 
- \int_t^T\! \!\! \llangle Du_s,\,(|\sigma_s|^2+|\bar\sigma_s|^2)Du_s\rrangle\,ds
+\frac{C}{\eps_1}\int_t^T  \!\!\!  \|u_s\|^2\, ds + {\eps_1}\int_t^T\!\!\!  \|Du_s\|^2\, ds
\\
&\text{(by ($\cA 1$))}\ \ \\
&\leq
-\kappa  \int_t^T\!\!  \|Du_s\|^2\,ds
+\frac{C}{\eps_1}\int_t^T \!\!\!  \|u_s\|^2\, ds + {\eps_1}\int_t^T\|Du_s\|^2\, ds
,\quad \eps_1>0.
\end{align*}
Using Schwartz inequality, we further have
\begin{align}
\int_t^T \!\!\!   \llangle u_s,\,  2 \sigma_sD\psi_s  \rrangle\,ds
\leq
\frac{K^2}{\eps_2}\int_t^T\|u_s\|^2\, ds
+{\eps_2}\int_t^T
\!\!\!
\|D\psi_s\|^2\, ds, \quad \eps_2>0. \label{est-psi-apriori}
\end{align}
In addition, we have
\begin{align*}
2E\left[ \sup_{\tau\in[t,T]}\bigg| \int_{\tau}^T  \llangle u_s,\, \psi_s\,dW_s \rrangle   \bigg|  \right]
&\leq 4 E \left[ \sup_{\tau\in[t,T]}\bigg| \int_t^{\tau}  \llangle u_s,\, \psi_s\,dW_s\rrangle\bigg|  \right]\\
\text{(by BDG inequality) }
&\leq 
C E\left[ \left(\int_t^T  \|u_s\|^2\|\psi_s\|^2   \,ds \right)^{1/2}\right].
\end{align*}
Notice $\lambda \in [0,1]$ and $$\frac{\kappa\lambda}{2}+ (1-\lambda) \geq 1 \wedge \frac{\kappa}{2}  >0.$$
Incorporating the above estimates and letting $\eps_1=\frac{\kappa}{2}$, we arrive at
\begin{align*}
& \delta E \left[  \sup_{s\in[t,T]} \|u_s\|^2 \right] + (1-\delta) E\left[ \|u_t\|^2 \right]
+E\int_t^T \left( \|\psi_s\|^2 +\|Du_s\|^2\right) \,ds\\
& \leq C\, E\Bigg[ \|G\|^2 
+\int_t^T\left( 1+\frac{1}{\eps_2}\right) \|u_s\|^2
+ \left|\llangle h_s,\, u_s  \rrangle \right|
+\eps_2\|D\psi_s\|^2\,ds\\
&\quad 
+\delta \left(\int_t^T  \|u_s\|^2\|\psi_s\|^2   \,ds \right)^{1/2}
\Bigg]\\
&\leq C\, E\Bigg[ \|G\|^2 
+\int_t^T\left( 1+\frac{1}{\eps_2}\right) \|u_s\|^2
+ \left|\llangle h_s,\, u_s  \rrangle \right|
+\eps_2\|D\psi_s\|^2\,ds
\\
&\quad 
+ \delta \int_t^T  \|\psi_s\|^2   \,ds \Bigg]
+\frac{\delta}{2} E \left[  \sup_{s\in[t,T]} \|u_s\|^2 \right],
\end{align*}
with $\delta \in \{0,1\}$. Applying Gronwall inequality successively for the cases $\delta=0$ and $\delta=1$, we obtain
\begin{align}
& E \left[  \sup_{s\in[t,T]} \|u_s\|^2 \right] 
+E\int_t^T \left(  \|\psi_s\|^2   +\|Du_s\|^2\right)\,ds    \nonumber\\
&\leq
C\, E\Bigg[ \|G\|^2 + \int_t^T
\!\!\!\!
 \left|\llangle h_s,\, u_s  \rrangle \right| \,ds
+\frac{1}{\eps_2} \int_t^T \!\!\!   \|u_s\|^2\,ds
+\eps_2\int_t^T  \!\!\! \|D\psi_s\|^2\,ds \Bigg], \label{est-step1}
\end{align}
with the constant $C$ depending only on $\kappa$, $K$ and $T$.


\tbf{Step 2.} Taking the spatial derivatives on both sides of BSPDE \eqref{BSPDE-NP-linear}, one can easily check that $(v,\Psi){\triangleq}(Du,D\psi)$ is a weak solution of the following Dirichlet problem\footnote{In view of Definition \ref{defn-solution}, a strong solution satisfies the associated BSPDE in the weak/distributional sense as a Sobolev space-valued random function. In fact, it always makes sense to differentiate a function in Sobolev space as the derivative can be well defined in the distributional sense; in particular, for the strong solution $(u,\psi)$, we have $(u,\psi)\in\mathcal H^1$ according to Definition \ref{defn-solution}, then it follows that $(Du,D\psi)\in\mathcal H$. Therefore, we take spatial derivatives and write the resulting equation in a straightforward way.}:
\begin{equation}\label{BSPDE-DP-linear}
  \left\{\begin{array}{l}
  \begin{aligned}
  -dv_t(x)=\, &\bigg[ \frac{\lambda}{2}\left( | \sigma_t(x)|^2+|\bar\sigma_t(x)|^2\right)D^2 v_t(x)
        + \lambda \sigma_t(x) D \Psi_t(x)\\
        &
         +\frac{\lambda}{2}D\left( | \sigma_t(x)|^2+|\bar\sigma_t(x)|^2\right)Dv_t(x)+ \lambda D\sigma_t(x)  \Psi_t(x) 
         \\
         &
         +\frac{1-\lambda}{2}D^2v_t(x)+Dh_t(x)
         \bigg]\, dt  
        -\Psi_t(x)\, dW_{t};\\
        v_t(0)=\,&0,\quad v_t(b)=0;\\
    v_T(x)=\,& DG(x).
    \end{aligned}
  \end{array}\right.
\end{equation}
Applying again It\^o's formula (Lemma \ref{lem-ito-formula} in the Appendix A) to the square norm yields
\begin{align*}
&\|v_t\|^2+\int_t^T\|\Psi_s\|^2\, ds-\|DG\|^2\\
&= \int_t^T \big \langle v_s,\, \lambda \big[
(|\sigma_s|^2+|\bar\sigma_s|^2)D^2v_s 
+ 2 \sigma_sD\Psi_s 
+D\left( | \sigma_s|^2+|\bar\sigma_s|^2\right)Dv_s
\\
&\quad\quad\quad\quad\quad
+2D\sigma_s  \Psi_s\big] +(1-\lambda) D^2u_s
    \big\rangle\,ds\\
&\quad +2\int_t^T \llangle v_s,  Dh_s\rrangle\,ds -2\int_t^T\llangle v_s,\, \Psi_s\,dW_s\rrangle,\quad \text{a.s. }\forall t\in[0,T].
\end{align*}
In view of the zero-Dirichlet condition, one has 
$$\int_t^T\llangle v_s,\,(1-\lambda) D^2v_s \rrangle = -(1-\lambda)\int_t^T \|Dv_s\|^2\,ds,$$
\begin{align}
&\int_t^T \Big\langle v_s,\,
(|\sigma_s|^2+|\bar\sigma_s|^2)D^2v_s 
+ 2 \sigma_sD\Psi_s 
+D\left( | \sigma_s|^2+|\bar\sigma_s|^2\right)Dv_s\nonumber \\
&\quad\quad\quad\quad
+2D\sigma_s  \Psi_s
    \Big\rangle\,ds
    				\nonumber \\
   & =
    -\int_t^T \llangle Dv_s,\,
(|\sigma_s|^2+|\bar\sigma_s|^2)Dv_s 
+ 2 \sigma_s\Psi_s 
    \rrangle\,ds
    			  \label{gradient-0}\\
    &\leq
    -\int_t^T \left( \|
\bar\sigma_s Dv_s\|^2
    -\eps_3 \|\sigma_sDv_s\|^2\right)
   \,ds
   + \frac{1}{1+\eps_3} \int_t^T \|\Psi_s\|^2\,ds
   										\label{gradient-1}\\
 & \text{ (by $(\cA 1)$ and $(\cA 2)$) }   \nonumber \\
&\leq
  -(\kappa-\eps_3K^2)\int_t^T
  \|\sigma_s Dv_s\|^2 
   \,ds
   + \frac{1}{1+\eps_3} \int_t^T \|\Psi_s\|^2\,ds,\quad \eps_3>0,
   \label{gradient-2}
\end{align}
and 
\begin{align*}
\int_t^T \llangle v_s,\,Dh_s
    \rrangle\,ds
    = -\int_t^T \llangle Dv_s,\,h_s
    \rrangle\,ds.
\end{align*}

Taking $\eps_3=\frac{\kappa}{2K^2}$ and in a similar way to \tbf{Step 1}, we get
\begin{align*}
& E \left[  \sup_{s\in[t,T]} \|v_s\|^2 \right] 
+E\int_t^T \left(    \|Dv_s\|^2+\|\Psi_s\|^2  \right)\,ds    
\\
&\leq
C\, E\Bigg[ \|DG\|^2 + \int_t^T \left|\llangle h_s,\, Dv_s  \rrangle \right| \,ds
 \Bigg],
\end{align*}
i.e.,
\begin{align*}
 &E \left[  \sup_{s\in[t,T]} \|Du_s\|^2 \right] 
+E\int_t^T  \left(    \|D^2u_s\|^2+\|D\psi_s\|^2  \right)\,ds    
\\
&\leq
C\, E\Bigg[ \|DG\|^2 + \int_t^T \left|\llangle h_s,\, D^2u_s  \rrangle \right| \,ds
 \Bigg],
\end{align*}
which, together with \eqref{est-step1}, implies
\begin{align*}
& E \left[  \sup_{s\in[t,T]} \|u_s\|_{H^{1,2}([0,b])}^2 \right] 
+E\int_t^T   \left(\|u_s \|^2_{H^{2,2}([0,b])}  +\| \psi_s\|_{H^{1,2}([0,b])}^2 \right) \,ds    \nonumber\\
&\leq
C\, E\Bigg[ \|G\|_{H^{1,2}([0,b])}^2 
 + \int_t^T \left(
 \left|\llangle h_s,\, u_s  \rrangle \right| +\left|\llangle h_s,\, D^2u_s  \rrangle \right|
  \right)\,ds\\&
+\frac{1}{\eps_2} \int_t^T\|u_s\|^2\,ds
+\eps_2\int_t^T \|D\psi_s\|^2\,ds \Bigg],
\end{align*}
with $C$ depending only on $\kappa$, $K$ and $T$.

Noticing that
\begin{align*}
 \int_t^T \!\!\!   \left(
 \left|\llangle h_s,\, u_s  \rrangle \right| +\left|\llangle h_s,\, D^2u_s  \rrangle \right|
  \right)\,ds
  \leq 
  \int_t^T\!\!\! \left(\frac{2}{\eps_2} \|h_s\|^2+2\eps_2\|u_s\|_{H^{2,2}([0,b])}^2       \right) ds
\end{align*}
and letting $\eps_2$ be small enough, one obtains for any $t\in[0,T]$,
\begin{align}
& E \left[  \sup_{s\in[t,T]} \|u_s\|_{H^{1,2}([0,b])}^2 \right] 
+E\int_t^T   \left(\|u_s \|^2_{H^{2,2}([0,b])}  +\| \psi_s\|_{H^{1,2}([0,b])}^2 \right) \,ds    \nonumber\\
&\leq
C\, E\Bigg[ \|G\|_{H^{1,2}([0,b])}^2 
 + \int_t^T \left(
 \left|\llangle h_s,\, u_s  \rrangle \right| +\left|\llangle h_s,\, D^2u_s  \rrangle \right|
  \right)\,ds
\Bigg]     \nonumber \\
&\leq
C\, E\Bigg[ \|G\|_{H^{1,2}([0,b])}^2 
 + \int_t^T \|h_s\|^2 \,ds
\Bigg] \label{est-prop-t}
\end{align}
with the constants $C$s depending on $\kappa$, $K$ and $T$. The uniqueness follows as an immediate consequence of the estimates and the linearity of the concerned BSPDE. 
\end{proof}

When $\lambda \sigma\equiv 0$,  estimate \eqref{est-psi-apriori} is not needed and we have 
\begin{cor}\label{cor-est-weak}
Let $(\cA 0),\,(\cA 1)$ and $(\cA 2)$ hold with $\lambda\sigma\equiv 0$,  and $h\in\cL^2_{\sF}(0,T;L^2([0,b]))$, $G\in L^2(\Omega,\sF_T;L^{2}([0,b]))$. Suppose $(u,\psi)$ is a weak solution of the Neumann problem \eqref{BSPDE-NP-linear}. Then the weak solution is unique and it holds that
\begin{align*}
\|(u,\psi)\|^2_{\cH}
& \leq C_1\,\left\{
\|G\|^2_{L^2(\Omega,\sF_T;L^{2}([0,b]))} + E\left[ \int_0^T \left| \llangle h_t, \, u_t\rrangle\right| 
\,dt \right]
\right\}\\
&\leq 
C_2\,\left\{
\|G\|^2_{L^2(\Omega,\sF_T;L^{2}([0,b]))} + \|h\|^2_{\cL^2_{\sF}(0,T;L^2([0,b]))}
\right\},
\end{align*}
where the constants $C_1$ and $C_2$ depend only on $\kappa$, $K$ and $T$.
\end{cor}

\begin{rmk}\label{rmk-linear}
When $\sigma$ is not vanishing, for a weak solution $(u,\psi)$, the estimate \eqref{est-psi-apriori} makes no sense. In fact, the term $\int_t^T \llangle u_s,\,  2 \sigma_sD\psi_s  \rrangle\,ds$ is not well-defined. Even when we apply the integration-by-parts formula, the function $ \psi$ has no intrinsic meaning on the boundary, nor does the term $u\sigma\psi$, because they are just restrictions to the boundary of $L^2([0,b])$ functions. Thus, for the Neumann problems like \eqref{BSPDE-NP} and \eqref{BSPDE-NP-linear}, It\^o's formula for the square norm is not applicable to the weak solutions when $\sigma$ is not vanishing, and this makes the existing methods for weak solutions inapplicable here.
\end{rmk}

\begin{rmk}\label{rmk-superparabolic}
If we explore in more detail the calculations  from \eqref{gradient-0} through \eqref{gradient-2},  we may see how the assumptions $(\cA 0)$ and $(\cA 1)$ respectively on adaptedness and superparabolicity contribute to the gradient estimates (\eqref{est-prop-t} for instance) and thus the improved regularity of solutions that is in demand to apply the It\^o-Kunita-Wentzell formula of Lemma \ref{lem-ito-wentzell} for the verification in the proof of Theorem \ref{thm-verification}. In particular, if there is only one Wiener process $W$ in the original control problem (see \eqref{min-contrl-probm} and \eqref{state-proces-contrl}) and we assume all the coefficients are adapted to the filtration $\sF$ generated by $W$, then the control problem turns out to be equivalent to the case when $\bar{\sigma}\equiv 0$ and this will make calculations  \eqref{gradient-0}-\eqref{gradient-2} and the gradient estimates in \eqref{est-prop-t} invalid. In this sense, it also explains why we have two Wiener processes $W$ and $B$ and use the coefficients associated with $B$ to construct the superparabolicity. 
\end{rmk}

\subsection{Existence and uniqueness of a strong solution (Proof of Theorem \ref{thm-Neumann-nonlinear})} \label{subsec:general nonlinear}


First, we consider the following Neumann problem with Laplacian operator:
\begin{equation}\label{BSPDE-NP-model}
  \left\{\begin{array}{l}
  \begin{aligned}
  -du_t(x)=\, &\left[ D^2 u_t(x)
       +h_t(x)
         \right]\, dt   -\psi_t(x)\, dW_{t}, 
        \\
        &\quad (t,x)\in [0,T]\times [0,b];\\
        Du_t(0)=\,&0,\quad Du_t(b)=0;\\
    u_T(x)=\,& G(x),  \quad x\in[0,b].
    \end{aligned}
  \end{array}\right.
\end{equation}

\begin{prop}\label{prop-model}
Let $$h\in\cL^2_{\sF}(0,T;L^2([0,b])) \quad \text{and}\quad G\in L^2(\Omega,\sF_T;H^{1,2}([0,b])).$$ BSPDE \eqref{BSPDE-NP-model} admits a unique strong solution $(u,\psi)$.
\end{prop}

\begin{proof} The uniqueness of strong solution follows directly from Proposition \ref{prop-apriori}. We need only to prove the existence.
\tbf{Step 1.} Suppose further $h\in \cL^2_{\sF}(0,T;H^{1,2}([0,b]))$ and $DG\in L^2(\Omega,\sF_T; H_0^{1,2}([0,b]))$.
By the theory on the Neumann problem of deterministic parabolic PDEs (see \cite[Theorem 7.20]{Lieberman}), there exists a unique strong solution $\hat{u}$ to PDE: 
\begin{equation}\label{BSPDE-NP-model-d}
  \left\{\begin{array}{l}
  \begin{aligned}
  -\partial_t \hat u_t(x)=\, & D^2 \hat u_t(x)
       +h_t(x), 
        \quad\quad \quad (t,x)\in [0,T]\times [0,b];\\
        D\hat u_t(0)=\,&0,\quad D\hat u_t(b)=0;\\
    \hat u_T(x)=\,& G(x),  \quad x\in[0,b],
    \end{aligned}
  \end{array}\right.
\end{equation}
such that $\hat u, D\hat u, D^2\hat u, \partial_t \hat u\in L^2(\Omega,\sF_T; L^2([0,T]\times [0,b]))$. Taking conditional expectations in Hilbert spaces (see \cite{DaPrato1992}), set
$$u_t=E\left[ \hat u_t \big| \sF_t\right],\quad \text{a.s., for each }t\in[0,T],$$
which admits a version in $\cS^2_{\sF}(0,T; L^2([0,b]))\cap\cL^2_{\sF}(0,T; H^{2,2}([0,b]))$ that together with $\psi\in\cL^2_{\sF}(0,T; L^2([0,b]))$ satisfies $L^2([0,b])$-valued BSDE:
\begin{equation}\label{BSDE-step1}
  \left\{\begin{array}{l}
  \begin{aligned}
  -du_t(x)=\, &\left[ D^2 u_t(x)
       +h_t(x)
         \right]\, dt   -\psi_t(x)\, dW_{t};\\
    u_T(x)=\,& G(x),  \quad x\in[0,b].
    \end{aligned}
  \end{array}\right.
\end{equation}
In view of the definition of $u$, $u$ satisfies the zero-Neumann boundary condition. By Definition \ref{defn-solution} and Corollary \ref{cor-est-weak}, it is easy to check that $(u,\psi)$ is the weak solution to BSPDE \eqref{BSPDE-NP-model}.

\tbf{Step 2.} We now prove that the constructed weak solution $(u,\psi)$ is in fact the unique strong solution of BSPDE \eqref{BSPDE-NP-model}. In a similar way to \tbf{Step 1}, it is easy to check that $D\hat u$ would be the strong solution of Dirichlet problem: 
\begin{equation}\label{BSPDE-NP-model-d2}
  \left\{\begin{array}{l}
  \begin{aligned}
  -\partial_t \hat v_t(x)=\, & D^2 \hat v_t(x)
       +Dh_t(x), 
        \quad\quad \quad (t,x)\in [0,T]\times [0,b];\\
        \hat v_t(0)=\,&0,\quad \hat v_t(b)=0;\\
    \hat v_T(x)=\,& DG(x),  \quad x\in[0,b],
    \end{aligned}
  \end{array}\right.
\end{equation}
and $(Du,D\psi)$ satisfies $L^2([0,b])$-valued BSDE \eqref{BSDE-step1} associated to the coefficients $(Dh, DG)$. In particular, we have 
$$(Du,D\psi)\in \cS^2_{\sF}(0,T;L^2([0,b])) \times \cL^2_{\sF}(0,T;L^2([0,b]))$$ 
and thus $(u,\psi)$ is the strong solution to BSPDE \eqref{BSPDE-NP-model}. For general $h\in\cL^2_{\sF}(0,T;L^2([0,b]))$ and $G\in L^2(\Omega,\sF_T; H^{1,2}([0,b]))$, we may choose a sequence $$\{(h^n,G^n)\}_{n\in\bN^+}\subset   \cL^2_{\sF}(0,T;H^{1,2}([0,b])) \times L^2(\Omega,\sF_T; H^{1,2}([0,b])) $$ with $\{DG^n\}_{n\in\bN^+} \subset L^2(\Omega,\sF_T; H_0^{1,2}([0,b]))$ such that $(h^n,G^n)$ converges to $(h,G)$ in  $\cL^2_{\sF}(0,T;L^2([0,b])) \times L^2(\Omega,\sF_T; H^{1,2}([0,b]))$. For each $(h^n,G^n)$, we get the corresponding strong solution $(u^n,\psi^n)$. Then the estimates in Proposition \ref{prop-apriori} yield the convergence of $(u^n,\psi^n)$ as well as the existence of strong solution. 
\end{proof}


Now, we may use the continuity method to prove Theorem \ref{thm-Neumann-nonlinear}. 
\begin{proof}[Proof of Theorem \ref{thm-Neumann-nonlinear}]
\tbf{Step 1.}
For each $h\in \cL^2_{\sF}(0,T;L^2([0,b]))$ and $\lambda \in [0,1]$, consider the following BSPDE
\begin{equation}\label{BSPDE-NP-lambda}
  \left\{\begin{array}{l}
  \begin{aligned}
  -du_t(x)=\, &\bigg\{\lambda \bigg[ \frac{1}{2}\left( | \sigma_t(x)|^2+|\bar\sigma_t(x)|^2\right)D^2 u_t(x)
        + \sigma_t (x)D \psi_t(x)\\&+\Gamma_t(x,u,Du,D^2u,\psi,D\psi)
         \bigg]   + (1-\lambda)D^2u_t(x)+h_t(x)
         \bigg\}\, dt  \\
    &    -\psi_t(x)\, dW_{t}, \quad (t,x)\in [0,T]\times [0,b];\\
        Du_t(0)=\,&0,\quad Du_t(b)=0;\\
    u_T(x)=\,& G(x),  \quad x\in[0,b].
    \end{aligned}
  \end{array}\right.
\end{equation}
Note that BSPDE \eqref{BSPDE-NP-nonlinear} corresponds to the special case when $\lambda=1$ and $h\equiv 0$.
We will generalize the a priori estimates from the linear case of Proposition \ref{prop-apriori} to nonlinear equation \eqref{BSPDE-NP-lambda}. Suppose $(u,\psi)$ is a strong solution of BSPDE \eqref{BSPDE-NP-lambda}. Applying Proposition \ref{prop-apriori} to each $t\in[0,T]$ (see also estimates \eqref{est-step1} and \eqref{est-prop-t}), we have by $(\cA 3)$
\begin{align}
& E \left[  \sup_{s\in[t,T]} \|u_s\|^2 \right] 
+E\int_t^T \|\psi_s\|^2 + \|Du_s\|^2  \,ds    \nonumber\\
&\leq
C\, E\Bigg[ \|G\|^2 + \int_t^T \left|\llangle h_s+\lambda\Gamma_s(\cdot,u,Du,D^2u,\psi,D\psi),\, u_s  \rrangle \right| \,ds
\nonumber\\
&\quad\quad\quad\quad
   +\frac{1}{\eps} \int_t^T\|u_s\|^2\,ds
+\eps\int_t^T \|D\psi_s\|^2\,ds \Bigg] \nonumber  \\
&\leq
C\, E\Bigg[ \|G\|^2 +\eps \int_t^T\left( \|h_s\|^2+ \left\|  \Gamma_s(\cdot,u,Du,D^2u,\psi,D\psi)\right\|^2+ \|D\psi_s\|^2\right)    \,ds
\nonumber\\
&\quad\quad\quad\quad
+\frac{3}{\eps} \int_t^T\|u_s\|^2\,ds
 \Bigg]   \nonumber \\
&\leq
C_1\, E\Bigg[ \|G\|^2 +\eps \int_t^T \left( \|\Gamma^0_s\|^2+\|h_s\|^2+\|u_s\|^2+\|Du_s\|^2+\|\psi_s\|^2  \right)    \,ds
\nonumber\\
&\quad\quad\quad\quad
+\frac{3}{\eps} \int_t^T\|u_s\|^2\,ds   
+\int_t^T \eps(1+\mu^2)\|D\psi_s\|^2+\eps\mu^2\|D^2u_s\|^2\,ds \Bigg] ,
\label{thm-est-1}
\end{align}
and
\begin{align}
& E \left[  \sup_{s\in[t,T]} \|u_s\|_{H^{1,2}([0,b])}^2 \right] 
+E\int_t^T   \left(\|u_s \|^2_{H^{2,2}([0,b])}  +\| \psi_s\|_{H^{1,2}([0,b])}^2 \right) \,ds    \nonumber\\
&\leq
C\, E\Bigg[ \|G\|_{H^{1,2}([0,b])}^2 
 + \int_t^T  \left|  \llangle h_s+  \lambda \Gamma_s(\cdot,u,Du,D^2u,\psi,D\psi),\,u_s\rrangle \right| \,ds  \nonumber \\
&\quad\quad\quad\quad +\int_t^T \left|  \llangle   h_s+\lambda \Gamma_s(\cdot,u,Du,D^2u,\psi,D\psi),\,D^2u_s\rrangle \right| \,ds
\Bigg] \nonumber\\
&\leq
C \, E\Bigg[ \|G\|_{H^{1,2}([0,b])}^2 
 \nonumber\\
&\quad\quad
 +\left(1+\frac{1}{\eps}\right) \int_t^T \left(\| \Gamma^0_s\|^2 +\|h_s\|^2+ \|u_s\|^2 + \|Du_s\|^2 +\|\psi_s\|^2     \right)\,ds \Bigg]   \nonumber \\
&\quad\quad  +\int_t^T E\left[  \left(\mu^2+\eps\right)\|D^2u_s\|^2 +(\mu^2+\eps) \|D\psi_s\|^2   \right]  \,ds ,
 \label{thm-est-2}
\end{align}
with $C$s depending on $\kappa,K,L$ and $T$. Letting $\eps<\frac{1}{2C_1+1}$, we have by \eqref{thm-est-1},
\begin{align}
& E \left[  \sup_{s\in[t,T]} \|u_s\|^2 \right] 
+E\int_t^T \left(\|\psi_s\|^2 + \|Du_s\|^2\right) \,ds    \nonumber\\
&\leq 
C_2\, E\Bigg[ \|G\|^2 
\nonumber\\
&\quad
+\int_t^T\!\!\!
\left(   \|\Gamma^0_s\|^2  +\|h_s\|^2 
+\frac{2}{\eps} \|u_s\|^2  + \eps(1+\mu^2)\|D\psi_s\|^2+\eps\mu^2\|D^2u_s\|^2\right)\,ds \Bigg], \label{thm-est-001}
\end{align}
with $C_2$ independent of $(\eps,\,\mu)$.
From \eqref{thm-est-001} and \eqref{thm-est-2}, it follows that, there exists $\mu_0$ depending on $\kappa$, $K$, $L$ and $T$ such that 
when $\mu<\mu_0$, letting $\eps$ be small enough and using Gronwall inequality yield
\begin{align}
&\|(u,\psi)\|_{\cH^1}
\nonumber \\
&\leq C\left(
\|G\|_{L^2(\Omega,\sF_T;H^{1,2}([0,b]))} 
+\left\|\Gamma^0 \right\|_{\cL^2_{\sF}(0,T;L^2([0,b]))}
+\|h\|_{\cL^2_{\sF}(0,T;L^2([0,b]))}
\right), \label{thm-est-apriori} 
\end{align}
 with the constant $C$ depending on $\mu$, $L$, $\kappa$, $K$ and $T$.

\tbf{Step 2.} Suppose $(u_1,\,\psi_1)$ and $(u_2,\,\psi_2)$ are two strong solutions of BSPDE \eqref{BSPDE-NP-lambda}. Then, the pair $(\delta u,\, \delta \psi) = (u_1-u_2,\,\psi_1-\psi_2)$ satisfies the following BSPDE:
\begin{equation}\label{BSPDE-NP-delta}
  \left\{\begin{array}{l}
  \begin{aligned}
  -d\delta u_t(x)=\, &\bigg\{\lambda \bigg[ \frac{1}{2}\left( | \sigma_t (x)|^2+|\bar\sigma_t (x)|^2\right)D^2 \delta u_t(x)
        + \sigma_t(x) D\delta  \psi_t(x)\\
        +\Gamma_t(&x,u_1,Du_1,D^2u_1,\psi_1,D\psi_1) 
         - \Gamma_t(x,u_2,Du_2,D^2u_2,\psi_2,D\psi_2)
         \bigg] \\&
         + (1-\lambda)D^2\delta u_t(x)
         \bigg\}\, dt  
        -\delta \psi_t(x)\, dW_{t} ;\\
        D\delta u_t(0)=\,&0,\quad D\delta u_t(b)=0;\\
    \delta u_T(x)=\,& 0.
    \end{aligned}
  \end{array}\right.
\end{equation}
Recalling $({\mathcal A} 3)$, we have
\begin{align*}
	&\left\|\Gamma_t(\cdot, u_1,Du_1,D^2u_1,\psi_1,D\psi_1 )  - \Gamma_t(\cdot, u_2,Du_2,D^2u_2,\psi_2,D\psi_2 )  \right\| \\
	& \leq 
	\mu \left( \left\|D^2(u_1-u_2)\right\| + \left\|D(\psi_1-\psi_2) \right\|\right) 
	\\
	&\quad+ L \left(  \left\|    u_1-u_2  \right\|_{H^{1,2}([0,b])}     +   \left\|    \psi_1-\psi_2  \right\|_{L^{2}([0,b])}  \right),
	\end{align*} 
a.s. for any $t\in[0,T]$.
In a similar way to \tbf{Step 1}, applying It\^o's formula (Lemma \ref{lem-ito-formula} in the Appendix A) to square norms of $(\delta u,\,\delta \psi)$, one gets estimates  \eqref{thm-est-1}, \eqref{thm-est-2} and \eqref{thm-est-001}, and further \eqref{thm-est-apriori} but with $(G,\Gamma^0,h)$ being replaced by zero values. This indicates the uniqueness of strong solution to BSPDE \eqref{BSPDE-NP-lambda} as well as to BSPDE \eqref{BSPDE-NP-nonlinear}.

\tbf{Step 3.} First, notice that the a priori estimate \eqref{thm-est-apriori} holds with the constant $C$ being independent of $\lambda\in [0,1]$. When $\lambda =0$, Proposition \ref{prop-model} implies that BSPDE \eqref{BSPDE-NP-lambda} admits a unique strong solution $(u,\psi)$. Noticing that when $\lambda=1$, BSPDE  \eqref{BSPDE-NP-lambda} coincides with \eqref{BSPDE-NP-nonlinear},   we then expect to extend the wellposedness of  BSPDE \eqref{BSPDE-NP-lambda} through the interval $ [0,1]$ starting from $\lambda =0$.  

Assume that for some $\lambda=\lambda_0$, BSPDE \eqref{BSPDE-NP-lambda}, satisfying assumptions $(\cA 0)-(\cA 3)$, admits a unique strong solution $(u,\psi)$, which is true when $\lambda_0=0$ (by Proposition \ref{prop-model} as above). Then, for each $(\check u,\check \psi)\in\cH^1$, the following BSPDE
\begin{equation*}
  \left\{\begin{array}{l}
  \begin{aligned}
  -du_t(x)=\, &\bigg\{\lambda_0 \bigg[ \frac{1}{2}\left( | \sigma_t(x)|^2+|\bar\sigma_t(x)|^2\right)D^2 u_t(x)
        + \sigma_t (x)D \psi_t(x)\\&+\Gamma_t(x,u,Du,D^2u,\psi,D\psi)
         \bigg]   \\
        +(\lambda -  &  \lambda_0) 
         \bigg[ \frac{1}{2}\left( | \sigma_t(x)|^2+|\bar\sigma_t(x)|^2\right)D^2 \check u_t(x)
        + \sigma_t (x)D \check \psi_t(x)
        \\&+\Gamma_t(x,\check u,D\check u,D^2\check u,\check \psi,D\check \psi) 
        -D^2\check u_t(x)
        \bigg]\\
         &\quad  + (1-\lambda_0)D^2u_t(x)\bigg\}\, dt  
      -\psi_t(x)\, dW_{t}, 
        \\& (t,x)\in [0,T]\times [0,b];\\
        Du_t(0)=\,&0,\quad Du_t(b)=0;\\
    u_T(x)=\,& G(x),  \quad x\in[0,b],
    \end{aligned}
  \end{array}\right.
\end{equation*}
is a special case of BSDPE \eqref{BSPDE-NP-lambda} with $\lambda=\lambda_0$ and
\[
\begin{split}
h_t(x)&=\left(\lambda -\lambda_0\right)\bigg[ \frac{1}{2}\left( | \sigma_t(x)|^2+|\bar\sigma_t(x)|^2\right)D^2 \check u_t(x)
        + \sigma_t(x) D \check \psi_t(x)\\&+\Gamma_t(x,\check u,D\check u,D^2\check u,\check \psi,D\check \psi) 
        -D^2\check u_t(x)
        \bigg],
\end{split}
\]
 and it has a unique strong solution $(u,\psi)$, and we can define the solution map as follows
\begin{align*}
\mathscr{M}_{\lambda_0}: \quad  \cH^1\rightarrow\cH^1,\quad\quad  (\check u,\,\check\psi) \mapsto (u,\,\psi).
\end{align*}
Then for any $(u_i,\,\psi_i)\in \cH^1$, $i=1,2$, in a similar way to \tbf{Step 2}, we have 
\begin{align*}
&\left\| (u_1-u_2,\,\psi_1-\psi_2)  \right\|_{\cH^1}
\\
&\leq C \big|  \lambda-\lambda_0  \big|\,
\bigg\|  
\frac{1}{2}\left( | \sigma|^2+|\bar\sigma|^2\right)D^2  (\check u_1-\check u_2)-D^2 (\check u_1-\check u_2)
        \\ &\quad\quad\quad  + \sigma D (\check \psi_1-\check \psi_2)
        +\Gamma_{\cdot}(\cdot,\check u_1,D \check u_1,D^2 \check u_1, \check \psi_1,D \check \psi_1) 
        \\
        &\quad\quad\quad -\Gamma_{\cdot}(\cdot,\check u_2,D \check u_2,D^2 \check u_2, \check \psi_2,D \check \psi_2)
\bigg\|_{\cL^2_{\sF}(0,T;L^2([0,b]))}
\\
&\leq \Tilde{C}\big|  \lambda-\lambda_0  \big|\,
\left\|\left(\check u_1-\check u_2,\,\check \psi_1-\check \psi_2\right)\right\|_{\cH^1},
\end{align*}
where the constant $\Tilde C$ does not depend on $(\lambda,\lambda_0)$. If $ \left| \lambda-\lambda_0  \right|<\frac{1}{\Tilde C}$, $\mathscr M_{\lambda_0}$ is a contraction mapping  and it has a unique fixed point $(u,\psi)\in\cH^1$ which is a strong solution of BSPDE \eqref{BSPDE-NP-lambda}. In this way, if BSPDE \eqref{BSPDE-NP-lambda} has a strong solution  for $\lambda_0$, so does it for any $\lambda$ satisfying $ \left| \lambda-\lambda_0  \right|<1/{\Tilde C}$. In finite steps starting from $\lambda=0$, we can reach $\lambda=1$, which together with the estimate \eqref{thm-est-apriori} and the uniqueness obtained in \textbf{Step 2} completes the proof.
\end{proof}


\section{Proof of Theorem \ref{thm-verification}} \label{subsec:proof}

We introduce an It\^o-Kunita-Wentzell formula for the composition of random fields and stochastic differential systems.

\begin{lem}\label{lem-ito-wentzell}
  Let \begin{equation}\label{SDE-wentzell}
X_t=x+\int_0^t\!\xi_r\,dr+\int_0^t dA_s+\int_0^t \rho_r\,dW_r+  \int_ 0^t\bar \rho_r\,dB_r  \quad 0\leq t\leq T,
\end{equation}
 with $(\xi,\rho,\bar{\rho})\in\cL^2_{\bar{\sF}}(0,T;\bR)\times \cL^2_{\bar{\sF}}(0,T;\bR^d) \times \cL^2_{\bar \sF} (0,T;\bR^d)$ and $A$ being an $\bar \sF_t$-adapted continuous bounded variation process satisfying $A_0=0$.
 Suppose $0\leq X_t\leq b$ a.s. and 
 \begin{align*}
    u_t(x)=u_0(x)+\int_0^t\!  q_r(x)\,dr+\int_0^t\!\psi_r(x)\,dW_r,\quad \textrm{for } (t,x)\in[0,T]\times[0,b],
  \end{align*}
  holds in the weak sense with $(u,q,\psi)$ in 
    \begin{align*}
   \left(\cS^2(0,T;H^{2,2}([0,b]))\cap\cL^2(0,T;H^{3,2}([0,b])) \right) \times \cL^2(0,T;H^{1,2}([0,b]))
   \\
   \times \cL^2(0,T;H^{2,2}([0,b])).    
   \end{align*}
   Then, for each $x\in[0,b]$, it holds almost surely that, for all $t\in[0,T]$,
 \begin{align}
     &u_t(X^{0,x}_t)- u_0(x) \label{eq-ito-kunita}  \\
     &
     =
     \!\int_0^t   \left[ q_r(X^{0,x}_r) +\xi_rDu_r(X^{0,x}_r)+\frac{1}{2}\left(|\rho|^2+|\bar\rho|^2\right)D^2u_r(X^{0,x}_r)+  \rho D\psi_r(X^{0,x}_r)  \right]\,dr    \nonumber \\
     &\quad+\int_0^tDu_r(X^{0,x}_r)\,dA_r+\int_0^t\left( \psi_r(X^{0,x}_r)    + Du_r(X^{0,x}_r)\rho_r\right)\,dW_r\nonumber \\
     &\quad +\int_0^t Du_r(X^{0,x}_r)\bar\rho_r\,dB_r. \nonumber
   \end{align}
\end{lem}

By Sobolev's embedding theorem,  $H^{m,2}(\bR)$ is continuously embedded into continuous function space $C^{m-1}$. Thus, the equation \eqref{eq-ito-kunita} makes sense for each $x\in[0,b]$ and in this way, Lemma \ref{lem-ito-wentzell} is similar to the first formula of Kunita \cite[Pages 118-119]{kunita1981some} if we replace the bounded domain $[0,b]$ by the whole real line $\bR$. To eliminate the affects of the boundary of the bounded domain, we extend the Sobolev spaces to the whole line, and the sketch of the proof is provided in the appendix.  We would note that in Lemma \ref{lem-ito-wentzell}, we consider the one-dimensional case for simplicity and that there is no essential difficulty in extending it to multi-dimensional cases.

A result on the Dirichlet problem of BSPDEs is introduced below, whose proof is the same to that of \cite[Theorem 3.1]{DuTang2010} under our assumptions.
\begin{lem}\label{lem-DT}
Consider the following Dirichlet problem of BSPDE;
\begin{equation}\label{BSPDE-DP-linear-prime}
  \left\{\begin{array}{l}
  \begin{aligned}
  -dv_t(x)=\, &\left[ \frac{1}{2}\left( | \sigma_t|^2+|\bar\sigma_t|^2\right)D^2 v_t(x)
        + \sigma_t D \psi_t(x)+h_t(x)
         \right]\, dt  \\
        &\quad-\psi_t(x)\, dW_{t}, 
        \quad\quad \quad (t,x)\in [0,T]\times [0,b];\\
        v_t(0)=\,&0,\quad v_t(b)=0;\\
    v_T(x)=\,& G(x),  \quad x\in[0,b],
    \end{aligned}
  \end{array}\right.
\end{equation}
with $G\in L^2(\Omega,\sF_T; H^{1,2}_0([0,b]))$ and $h\in \cL^2(0,T;L^2([0,b]))$. Under assumptions $(\cA 0), \, (\cA 1)$ and $(\cA 2)$ {with $\sigma_t(0)=\sigma_t(b)=0$ a.s. for any $t\in[0,T]$}, BSPDE \eqref{BSPDE-DP-linear-prime} admits a unique \textit{weak} solution $(v,\psi)$ which is also the unique \textit{strong} solution with
\begin{align}
\|(v,\psi)\|_{\cH^1}\leq C \left(
\|G\|_{L^2(\Omega,\sF_T;H^{1,2}([0,b]))} + \|h\|_{\cL^2(0,T;L^2([0,b]))}
\right)
\end{align}
with the constant $C$s depending on $\kappa$, $K$ and $T$.
\end{lem}

\begin{proof}[Proof of Theorem \ref{thm-verification}]
 We first reduce the Neumann problem \eqref{BSPDE-NP} to the case with zero Neumann boundary condition.
In view of assumption (i) of $(\cA^*)$ and Definition \ref{defn-solution},  setting 
$$
\left(\hat{g}_t,\hat{\mathscr G}_t,\hat \cG_t\right)(x)=\int_0^x \left(  {g}_t, {\mathscr G}_t, \cG_t\right)(y)\,dy, \quad  \text{for each }\,(t,x)\in[0,T]\times[0,b],
$$
we have $(\hat g,\,\hat\cG)$ is the strong solution of the following BSPDE
\begin{equation}\label{BSPDE-NP-g}
  \left\{\begin{array}{l}
  \begin{aligned}
  -d\hat g_t(x)=\, &\left[ \frac{1}{2}\left( | \sigma|^2+|\bar\sigma|^2\right)D^2 \hat g_t(x)
        + \sigma D \hat \psi_t(x)+  \hat f_t(x)       \right]\, dt  \\
        &\quad -\hat \psi_t(x)\, dW_{t}, 
        \quad (t,x)\in [0,T]\times [0,b];\\
        D\hat g_t(0)=\,&g_t(0),\quad D \hat g_t(b)=g_t(b);\\
    \hat g_T(x)=\,& \int_0^xg_T(y)dy,  \quad x\in[0,b],
    \end{aligned}
  \end{array}\right.
\end{equation}
with $\hat f\in \cL^2(0,T;H^{1,2}([0,b]))$ being defined
$$\hat{f}_t(x) =   -\frac{1}{2}\left( | \sigma|^2+|\bar\sigma|^2\right)D^2 \hat g_t(x)
        - \sigma D \hat \psi_t(x) +\hat{\mathscr G}_t(x).$$
Thus, the existence and uniqueness of strong solution $(u,\psi)$ to BSPDE \eqref{BSPDE-NP} is equivalent to that of the strong solution $(\tilde{u},\tilde{\psi})=(u-\hat g,\psi-\hat\cG)$ to the following BSPDE:
\begin{equation}\label{BSPDE-NP-zero}
  \left\{\begin{array}{l}
  \begin{aligned}
  -d\tilde {u}_t(x)=\, &\bigg[ \frac{1}{2}\left( | \sigma|^2+|\bar\sigma|^2\right)D^2 \tilde{u}_t(x)
        + \sigma D \tilde \psi_t(x)\\
        &\quad +{\mathbb H} _t(x,D\tilde u_t(x)+D\hat g_t(x))-\hat{f}_t(x)
         \bigg]\, dt  \\
    &\quad    -\tilde \psi_t(x)\, dW_{t}, 
        \quad\quad (t,x)\in [0,T]\times [0,b];\\
        D\tilde u_t(0)=\,&0,\quad D\tilde u_t(b)=0;\\
    \tilde u_T(x)=\,& G(x)-\hat g_T(x),  \quad x\in[0,b].
    \end{aligned}
  \end{array}\right.
\end{equation}

By Theorem \ref{thm-Neumann-nonlinear}, BSPDE \eqref{BSPDE-NP-zero} has a unique strong solution $(\tilde u, \tilde \psi)$. Taking derivatives, one can easily check that 
$$(v,\zeta)\triangleq (D\tilde u, D\tilde \psi)=(Du-D\hat g,D\psi-D\hat\cG)=(Du- g,D\psi-\cG)$$
 is a weak solution of the following Dirichlet problem:
\begin{equation}\label{BSPDE-DP-zero}
  \left\{\begin{array}{l}
  \begin{aligned}
  -dv_t(x)=\, &\left[ \frac{1}{2}\left( | \sigma|^2+|\bar\sigma|^2\right)D^2 v_t(x)
        +\sigma D \zeta_t(x)+F_t(x)
         \right]\, dt  \\
    &\quad    -\zeta_t(x)\, dW_{t}, 
        \quad\quad (t,x)\in [0,T]\times [0,b];\\
        v_t(0)=\,&0,\quad v_t(b)=0;\\
    v_T(x)=\,& DG(x)- g_T(x),  \quad x\in[0,b],
    \end{aligned}
  \end{array}\right.
\end{equation}
with 
\begin{align*}
F_t(x)&=
- D\hat{f}_t(x)
        +\left(D\sigma \sigma' + D\bar\sigma\sigma'\right) D^2\tilde u_t(x) + D \sigma D\tilde \psi_t(x)
         \\
         &\quad+(D{\mathbb H} )_t(x,D\tilde u_t(x)+g_t(x)) \\
         &\quad+(\partial_v{\mathbb H} )_t(x,D\tilde u_t(x)+g_t(x)) \left(D^2\tilde u_t(x)+Dg_t(x)\right).
\end{align*}
By assumption (ii) of $(\cA^*)$, one has $F\in\cL^2(0,T;L^2([0,b]))$. Then, one can conclude from Lemma \ref{lem-DT} that  $(v,\zeta)$ turns out to be a strong solution. Thus,
$(D\tilde u, D\tilde \psi)=(v,\zeta)\in\cH^1$ and moreover, $(D u, D \psi)=(D\tilde u+g, D\tilde \psi+\cG)\in \cH^1$.  Hence, $(u,\psi)\in \cH^2$. This regularity and assumption (iii) of $(\cA^*)$ indicate the admissibility of the control $\theta^*_t=\Pi_t(X_t^*,u_t(X_t^*))$. For each admissible control $\theta$,  applying the generalized It\^o-Kunita-Wentzell formula to $u_t(X_t^{0,x;\theta})$ indicates that for each $x\in [0,b]$, there holds almost surely
\begin{align}
&u_t(X_t^{0,x;\theta})\nonumber\\
&=E\left[
\int_t^T\essinf_{\tilde \theta\in\Theta}\left\{ \beta_r(X_r^{0,x;\theta},\tilde \theta)Du_r(X_r^{0,x;\theta}) +f_r(X_r^{0,x;\theta},\tilde \theta)   \right\} 
\Big|\bar\sF_t\right]\nonumber\\
&\quad +E\left[ - \int_0^{T}\beta_r(X_r^{0,x;\theta_r},\theta_r)Du_r(X_r^{0,x;\theta})\,dr
+G(X_T^{0,x;\theta})  \big|\bar\sF_t\right] \nonumber  \\
&\leq E\left[G(X_T^{0,x;\theta})+
\int_t^T f_r(X_r^{0,x;\theta}, \theta_r)   \,dr
\Big|\bar\sF_t\right],\quad  \forall\,\,t\in[0,T].
\label{ineq-verification-prf-1}
\end{align}
Thus, for any admissible control $\theta$, it holds almost surely,
\begin{align}
u_t(X_t^{0,x;\theta})\leq J_t(X_t^{0,x;\theta};\theta), \quad \forall\,\,t\in[0,T].  \label{ineq-verification-proof}
\end{align}
On the other hand, in a similar way to \eqref{ineq-verification-prf-1}, for each $x\in [0,b]$,
\begin{align}
&u_t(X_t^{*})=E\left[G(X_T^{*})+
\int_t^T f_r(X_r^{*}, \theta^*_r)   \,dr
\Big|\bar\sF_t\right]=J_t(X_t^*;\theta^*)\label{eq-verification-prf}
\end{align}
holds for all $t\in[0,T]$ with probability 1. Hence, in view of  relations \eqref{ineq-verification-proof} and \eqref{eq-verification-prf}, $u_t(x)$ coincides with the value function and the optimal control is given by $\theta^*_t=\Pi_t(X^*_t,Du_t(X^*_t))$ with the optimal state process $X^*_t$ satisfying RSDE \eqref{state-proces-contrl-optimal}. We complete the proof.
\end{proof}

\begin{rmk}
It is worth noting that when the dimension is bigger than one, simply taking derivatives does not arrive at a Dirichlet problem for the gradients. In other words, the method used in the above proof can not be directly extended to multidimensional cases. 
\end{rmk}


\begin{appendix}

\section{An It\^o formula for the square norm of solutions of SPDEs}
Let $(V,\|\cdot\|_V)$ be a real reflexive and separable Banach space, and $H$ a real
separable Hilbert space. With a little notational confusion, the inner product and norm in $H$ is denoted by $\langle \cdot,\ \cdot\rangle $ and
$\|\cdot\|$ respectively. Assume that $V$ is densely and continuously imbedded in $H$. Thus, the dual space $H'$ is also
continuously imbedded in $V'$ which is the dual space of $V$. Simply, we denote the
above framework by
\begin{equation*}
  V\hookrightarrow H \cong H'\hookrightarrow V'.
\end{equation*}
We denote by $\|\cdot\|_*$ the norm in $V'$. The dual product between $V$
and $V'$ is denoted by $\langle \cdot,\ \cdot \rangle _{V',V}$.  $(V,H,V')$ is called a Gelfand triple.

The It\^o formula plays a crucial role in the theory of SPDEs (see \cite{Krylov_Rozovskii81,RenRocknerWang2007} for instance). In the following, we introduce a backward version; see \cite[Theorem 3.2]{QiuTangBDSDES2010} for the proof for general cases.

\begin{lem}\label{lem-ito-formula}
  Let $\xi\in  L^2(\Omega,\sF_T,H)$, $F\in \cL^2(0,T;V')$ and $(u,\psi)\in \cL^2(0,T;V)\times \cL^2(0,T;L(\bR^d,H))$ with $\left(L(\bR^d,H),\|\cdot\|_1,\langle\cdot,\,\cdot\rangle_1\right)$ being the space of Hilbert-Schmidt operators from $\bR^d$ to $H$. Assume that the following backward SDE
  \begin{equation}\label{BDSES trivialcase}
    u_t=\xi+\int_t^T F_s\ ds-\int_t^T \psi_s\, dW_s,\ t\in
    [0,T],
      \end{equation}
  holds in the weak sense, i.e., for any $\phi\in V$, it holds almost surely that
     \begin{equation*}
    \langle u_t,\,\phi\rangle =\langle \xi,\, \phi\rangle +\int_t^T  
    \langle F_s,\,\phi\rangle_{V',V}\ ds-\int_t^T \llangle \phi,\,\psi_s\ dW_s\rrangle,\ t\in
    [0,T].
  \end{equation*}
  Then we assert that $u\in \cS^2(0,T;H)$ and the following It\^o formula
  holds almost surely
  \begin{equation}\label{Ito formula square}
    \begin{split}
        \|u_t\|^2=\ &\|\xi\|^2+\int_t^T\left( 2\langle F_s,\ u_s\rangle_{V',V}
                    -\|\psi_s\|_1^2\right)\ ds
                    \\
                    &-\int_t^T2\llangle u_s,\ \psi_s\ dW_s\rrangle,\quad t\in[0,T].
    \end{split}
  \end{equation}
\end{lem}
\begin{rmk}\label{rmk-ito-formula}
In Lemma \ref{lem-ito-formula}, suppose additionally $\tilde \xi\in  L^2(\Omega,\sF_T,H)$, $\tilde F\in \cL^2(0,T;V')$, $(\tilde u,\tilde \psi)\in \cL^2(0,T;V)\times \cL^2(0,T;L(\bR^d,H))$ and 
\begin{equation}\label{BDSES trivialcase-1}
    \tilde u_t=\xi+\int_t^T  \tilde F_s\ ds-\int_t^T \tilde \psi_s\, dW_s,\ t\in
    [0,T],
      \end{equation}
  holds in the weak sense. Then $\tilde u\in \cS^2(0,T;H)$, and applying the parallelogram rule yields that the following holds almost surely
  \begin{equation}\label{Ito formula square-pro}
    \begin{split}
        \llangle u_t,\,\tilde u_t\rrangle =\ &\llangle \xi,\,\tilde \xi\rrangle 
        +\int_t^T\left( \langle F_s,\,\tilde u_s\rangle_{V',V} + \langle \tilde F_s,\, u_s\rangle_{V',V}
                    -\langle \psi_s,\,\tilde\psi \rangle_1^2\right)\ ds\\
                    &\, -\int_t^T\llangle \tilde u_s,\ \psi_s\ dW_s\rrangle
                    -\int_t^T\llangle  u_s,\ \tilde \psi_s\ dW_s\rrangle,\quad \forall \, t\in[0,T].
    \end{split}
  \end{equation}
\end{rmk}

\section{Proof of Lemma \ref{lem-ito-wentzell}}

\begin{proof}[Sketch of the proof of Lemma \ref{lem-ito-wentzell}]
The Sobolev space theory allows us to  extend $H^{k,2}([0,b])$ to $H^{k,2}(\bR)$ for integers $k\geq 1$. In particular, when $k=1,2$, the bounded linear extension operator can be constructed (as in \cite[Pages 254-257]{Evans-1998-PDE}) as follows: for each $\zeta\in H^{1,2}([0,b])$ or $\zeta\in H^{2,2}([0,b])$,
\begin{equation*}
\mathcal E \zeta(x)\triangleq
  \left\{\begin{array}{l}
  \begin{aligned}
  &\zeta(x),  \,&&\text{if } x\in [0,b];  \\
        &\gamma(x)\left[ -3\zeta(-x)+4\zeta(-x/2) \right], \,&&\text{if }x\in [-b,0];\\
        &\gamma(x)\left[ -3\zeta(2b-x)+4\zeta((2b-x)/2)  \right],\, && \text{if }x\in[b,2b];\\
        &0,\, &&\text{if }x\in(-\infty,-b)\cup(2b,\infty),
    \end{aligned}
  \end{array}\right.
\end{equation*}
where $\gamma\in C_c^{\infty}(\bR)$ satisfying $\gamma(x)=1$ for $x\in[0,b]$ while $\gamma(x)=0$ for $x\in(-\infty,-b/2]\cup[3b/2,\infty)$. $\mathcal E \zeta$ is called an extension of $\zeta$ to $\bR$. 
   Then it is easy to check that
 \begin{align}
    \mathcal E u_t(x)=\mathcal Eu_0(x)+\int_0^t\!  \mathcal E q_r(x)\,dr+\int_0^t\mathcal E \psi_r(x)\,dW_r,\quad \textrm{for } (t,x)\in[0,T]\times\bR, \label{eq-lem-ito-kunita-1}
  \end{align}
holds in the weak sense.

Define   
\begin{equation}
    \phi(x)=
    \begin{cases}
      \tilde{c}e^{\frac{1}{x^2-1}}&\quad \textrm{if } |x|\leq 1;\\
      0 &\quad \textrm{otherwise};
    \end{cases}
    \quad \mbox{with} \quad \tilde{c}:=\left(\int_{-1}^1e^{\frac{1}{x^2-1}}\,dx\right)^{-1},
  \end{equation}
  and for $l\in\bN$, set $\phi_l(x)=l\phi(lx)$, $x\in\bR$. 
  It\^o's formula yields that, for each $y\in\bR$,
  \begin{align*}
    &\phi_l(X_t^{0,x}-y)\\
     & = \phi_l(x-y)+\int_0^t\!D\phi_l(X_r^{0,x}-y)\rho_r\,dW_r+\int_0^t D\phi_l(X_r^{0,x}-y)\bar\rho_r\,dB_r\\
    &\quad+
    \int_0^t \!  \left[D\phi_l(x_r^{s,x}-y)\xi_r +\frac{1}{2}\left( |\rho_r|^2+|\bar \rho_r|^2  \right)D^2\phi_l(x_r^{s,x}-y)\right] \,dr\\
    &\quad+
    \int_0^t D\phi_l(x_r^{s,x}-y) \,dA_r ,\quad t\in[s,T].
  \end{align*}
 In view of \eqref{eq-lem-ito-kunita-1}, we have by It\^o's formula of Remark \ref{rmk-ito-formula}, 
  \begin{align*}
      &  \int_{\bR}\!\!\phi_l(X_t^{0,x}-y) \mathcal E u_t(y)\,dy- \int_{\bR}\!\!\phi_l(x-y)\mathcal E u_0(y)\,dy\\
      &\quad\quad-\int_0^t\!\!\int_{\bR}\! D \mathcal E u_r(y) \phi_l(X_r^{0,x}-y)\,dy dA_r
      \\
    &=
   \int_0^t\!\!\int_{\bR}\!
      \Big[D\mathcal E u_r(y)\xi_r
      +\mathcal E q_r(y)  + \frac{1}{2}\left( |\rho_r|^2+|\bar \rho_r|^2  \right)D^2\mathcal E u_r(y) 
      \\
      &\quad\quad\quad \quad
      +\rho_rD \mathcal E u_r(y) \Big]\phi_l(X_r^{0,x}-y)\,dydr\\
    &\quad +\int_0^t\!\!\int_{\bR}\!\phi_l(X_r^{0,x}-y)\left( \rho_rD\mathcal E u_r(y)+\psi_r(y)\right)\,dydW_r
      \\
      &\quad +\int_0^t\!\!\int_{\bR}\!\phi_l(X_r^{0,x}-y)\bar \rho_rD \mathcal E u_r(y)\,dydB_r,
  \end{align*}
a.s. for all $t\in[0,T]$. We note that as $X^{s,x}_{\cdot}\in\cS_{\bar{\sF}}^2(0,T;[0,b])$, all the above integrals on $\bR$ are taken on a compact set for almost every $\omega\in\Omega$ and thus make sense. Since the sequence of convolutions indexed by $l$ approximates to the identity and $0\leq X_t\leq b$ a.s., letting $l\rightarrow \infty$ and recalling that $\mathcal E u_t(x)=u_t(x)$ for $(t,x)\in[0,T]\times[0,b]$, we obtain that for each $x\in[0,b]$, it holds almost surely that, for all $t\in[0,T]$,
 \begin{align*}
     &u_t(X^{0,x}_t) 
     \\
     &= u_0(x)
    +\!\int_0^t   \bigg[ q_r(X^{0,x}_r) +\xi_rDu_r(X^{0,x}_r)
    +\frac{1}{2}\left(|\rho_r|^2+|\bar\rho_r|^2\right)D^2u_r(X^{0,x}_r)\\
    &\quad\quad\quad\quad\quad+  \rho_r D\psi_r(X^{0,x}_r)  \bigg]\,dr     
     +\int_0^t Du_r(X^{0,x}_r)\,dA_r
        \\
   &\quad
     +\int_0^t\left( \psi_r(X^{0,x}_r)    + Du_r(X^{0,x}_r)\rho_r\right)\,dW_r
  +\int_0^t Du_r(X^{0,x}_r)\bar\rho_r\,dB_r.
   \end{align*}
\end{proof}

\end{appendix}


\bibliographystyle{siam}
\bibliography{ref_qjn}

\end{document}